\documentclass[a4paper,10pt]{amsart}
\pdfoutput=1
\usepackage[utf8]{inputenc}
\usepackage{amsmath,amsthm,amssymb,mathrsfs,thmtools,leftidx}
\usepackage{dsfont}
\newtheoremstyle{mystyle}{}{}{}{}{\bf}{}{\newline}{}
\theoremstyle{mystyle}

\usepackage{enumitem}
\usepackage[toc,page]{appendix}
\usepackage{mathtools,array}
\usepackage[pagewise,displaymath, mathlines]{lineno} 

\usepackage{graphicx,lipsum}
\usepackage{fancyhdr,floatpag}
\usepackage{upgreek,color}
\usepackage[colorlinks=false,bookmarks=true]{hyperref}
\usepackage[usenames, dvipsnames]{xcolor}
\usepackage[style=numeric,backend=bibtex, bibencoding=ascii, maxnames=5]{biblatex}
\usepackage{blkarray}
\usepackage{multirow}
\RequirePackage{filecontents}  
\hypersetup{linkbordercolor={0 0.60 0.75}, linkcolor=MidnightBlue}
\hypersetup{citebordercolor={0 0.60 0.75}, citecolor=MidnightBlue}

\renewbibmacro{in:}{}
\numberwithin{equation}{section}
\theoremstyle{plain}

\newtheorem{thm}{Theorem}[section]

\newtheorem{dfn}[thm]{Definition}
\newtheorem{lem}[thm]{Lemma}

\newtheorem{pps}[thm]{Proposition}

\theoremstyle{remark}
\newtheorem{rmk}[thm]{Remark}
\newtheorem{examplex}[thm]{Example}
\newenvironment{exm}
  {\pushQED{\qed}\renewcommand{\qedsymbol}{$\triangle$}\examplex}
  {\popQED\endexamplex}

\declaretheoremstyle[
  spaceabove=-6pt,
  spacebelow=6pt,
  headfont=\normalfont\itshape,
  postheadspace=1em,
  qed=\qedsymbol,
  headpunct={}
]{mystyle}

\newcommand{\Z}{\mathds{Z}}
\newcommand{\C}{\mathds{C}}
\newcommand{\Q}{\mathds{Q}}
\newcommand{\F}{\mathds{F}}
\newcommand{\N}{\mathds{N}}
\newcommand{\R}{\mathds{R}}
\newcommand{\G}{\mathbf{G}}
\newcommand{\m}{\mathbf{m}}
\newcommand{\p}{\mathfrak{p}}
\newcommand{\Pp}{\mathfrak{P}}
\newcommand{\ri}{\mathcal{O}}
\newcommand{\lri}{\mathfrak{o}}
\newcommand{\Lri}{\mathfrak{O}}
\newcommand{\g}{\mathfrak{g}}

\newcommand{\lrip}{\lri/\p^N}
\newcommand{\ric}{\ri_{\p}}
\newcommand{\W}{W^{\lri}}
\newcommand{\f}{\mathfrak{f}}
\newcommand{\z}{\mathfrak{z}}
\newcommand{\Hom}{\text{Hom}}
\newcommand{\pints}{\mathscr{Z}}
\newcommand{\pint}{\mathscr{Z}^{\lri}}
\newcommand{\Pint}{\mathscr{Z}^{\Lri}}
\newcommand{\cczeta}{\zeta^\textup{cc}}
\newcommand{\rzeta}{\zeta^\textup{irr}}
\newcommand{\rtzeta}{\zeta^{\widetilde{\textup{irr}}}}
\newcommand{\kzeta}{\zeta^\textup{k}}
\newcommand{\rg}{\widetilde{r}}

\newcommand{\lvlzf}{\mathcal{Z}}
\newcommand{\rlvlzf}{\mathcal{Z}^\textup{irr}}
\newcommand{\clvlzf}{\mathcal{Z}^\textup{cc}}
\newcommand{\astlvlzf}{\lvlzf^{\ast}}
\newcommand{\astzeta}{\zeta^{\ast}}
\newcommand{\ddkzeta}{\zeta_{{K}}}
\newcommand{\rk}{\textup{rk}}
\newcommand{\rad}{\textup{Rad}(B_{\omega}^{N})}
\newcommand{\vry}{\nu(\mathcal{R}(\mathbf{y}))}
\newcommand{\re}{\text{Re}}

\newcommand{\gn}{\mathfrak{g}_{N}}
\newcommand{\gnc}{\mathfrak{g}'_{N}}
\newcommand{\zn}{\mathfrak{z}_{N}}

\newcommand{\Nr}{\mathcal{N}^{\mathcal{R}}}
\newcommand{\Na}{\mathcal{N}^{A}}
\newcommand{\Nb}{\mathcal{N}^{B}}

\newcommand{\ua}{u_{A}}
\newcommand{\ub}{u_{B}}
\newcommand{\ur}{u_{\mathcal{R}}}

\newcommand{\cc}{{c}}

\newcommand{\irrup}{\textup{irr}}
\newcommand{\ccup}{\textup{cc}}
\newcommand{\ad}{\textup{ad}}
\newcommand{\Ker}{\textup{Ker}}

\newcommand{\specp}{\textup{Spec}(\ri)\setminus\{(0)\}}
\newcommand{\Gl}{\textup{GL}}
\newcommand{\Irr}{\textup{Irr}}
\newcommand{\Frac}{\textup{Frac}}

\newcommand{\Dr}{\mathcal{D}}
\newcommand{\Dl}{\mathcal{D}_{\leq}}
\newcommand{\Poin}{\mathcal{P}^{\lri}}
\makeatletter
\newcommand{\thickhline}{
    \noalign {\ifnum 0=`}\fi \hrule height 1pt
    \futurelet \reserved@a \@xhline
}
\newcolumntype{"}{@{\hskip\tabcolsep\vrule width 1pt\hskip\tabcolsep}}
\makeatother

\makeatletter
\renewcommand*\env@matrix[1][*\c@MaxMatrixCols c]{
  \hskip -\arraycolsep
  \let\@ifnextchar\new@ifnextchar
  \array{#1}}
\makeatother

\DeclareMathOperator{\Mat}{Mat}
\DeclareMathOperator{\class}{k}

\makeatletter
\renewcommand{\@secnumfont}{\bfseries}
\makeatother
\usepackage{etoolbox}
\makeatletter
\patchcmd{\section}{\scshape}{\bf}{}{}
\patchcmd{\subsubsection}{\scshape}{\bf}{}{}
\makeatother

\title[paper]{Bivariate representation and conjugacy class zeta functions associated to unipotent group schemes, I: Arithmetic properties \vspace{-3ex}}
\author{Paula Macedo Lins de Araujo}
\address{Fakult\"at f\"ur Mathematik, Universit\"at Bielefeld, Postfach 100131, 33501 Bielefeld, Germany}
\email{lins@math.uni-bielefeld.de}
\date{\today}

\begin{document}
\begin{abstract}
This is the first of two papers in which we introduce and study two bivariate zeta functions associated to unipotent group schemes over rings of integers of number fields. 
One of these zeta functions encodes the numbers of isomorphism classes of irreducible complex representations of finite dimensions of 
congruence quotients of the associated group and the other one encodes the numbers of conjugacy classes of each size of such quotients. 
In this paper, we show that these zeta functions satisfy Euler factorisations and almost all of their Euler factors are rational and satisfy  functional equations.
Moreover, we show that such bivariate zeta functions specialize to (univariate) class number zeta functions. 
In case of nilpotency class~$2$, bivariate representation zeta functions also specialize to (univariate) twist representation zeta functions.
\end{abstract}
\maketitle \vspace{-0.5cm}
\thispagestyle{empty}

\section{Introduction and statement of main results}
\subsection{Introduction}\label{intro}
Let~$G$ be a group and, for $n \in \N$, write
\begin{align*}
  r_n(G)&=|\{\text{isomorphism classes of $n$-dimensional irreducible complex}\\
  &\phantom{r(G)}\text{representations of } G\}|,\\
  c_n(G)&=|\{\text{conjugacy classes of $G$ of cardinality }n\}|.
\end{align*}
If~$G$ is a topological group, we only consider continuous representations. 

We study the bivariate zeta functions of groups associated to unipotent group schemes encoding either the numbers $r_n(Q)$ or the numbers $c_n(Q)$ of certain finite quotients~$Q$ of the infinite groups considered. 
We first recall the definitions of (univariate) representation and conjugacy class zeta functions.
\begin{dfn}\label{univa} Let~$G$ be a group and $s$ a complex variable.
\begin{enumerate} 
\item If all~$r_n(G)$ are finite, then the \emph{representation zeta function} of~$G$ is
\[\rzeta_{G}(s)=\sum_{n=1}^{\infty}r_n(G)n^{-s}. \]
\item If all~$c_n(G)$ are finite, then the \emph{conjugacy class zeta function} of~$G$ is 
\[\cczeta_{G}(s)=\sum_{n=1}^{\infty}\cc_n(G)n^{-s}.\]
\end{enumerate}
\end{dfn}
The groups considered in the present paper are groups associated to unipotent group schemes which are obtained from nilpotent Lie lattices; see Section~\ref{groups}. 
From here on, let~$K$ denote a number field and~$\ri$ its ring of integers. 
Let~$\G$ be a unipotent group scheme over~$\ri$. The group~$\G(\ri)$ is a finitely generated, torsion-free nilpotent group ($\mathcal{T}$-group for short); 
see~\cite[Section~2.1.1]{StVo14}.

We observe that for a $\mathcal{T}$-group~$G$ the numbers $r_n(G)$ and $c_n(G)$ are not all finite. For this reason, one cannot define representation and conjugacy class zeta functions of~$G$ as in Definition~\ref{univa}. 
In the representation case, many authors have overcome this by considering zeta functions encoding the $n$-dimensional irreducible complex representations of such groups up to tensoring by one-dimensional representations; see Section~\ref{app2}. 

Our idea is to investigate zeta functions encoding the relevant data $r_n(Q)$ or $c_n(Q)$ of principal congruence quotients~$Q$ of~$\G(\ri)$. We define bivariate complex functions: one variable concerns either the dimensions of the representations considered or the cardinalities of the conjugacy classes considered, and the other variable concerns the principal congruence subgroups. 
\begin{dfn}\label{biva} The \emph{bivariate representation} and the \emph{bivariate conjugacy class zeta functions} of $\G(\ri)$ are
\begin{align*} 
\rlvlzf_{\G(\ri)}(s_1,s_2)&=\sum_{(0) \neq I \unlhd \ri }\rzeta_{\G(\ri/I)}(s_1)|\ri:I|^{-s_2}\text{ and}\\
\clvlzf_{\G(\ri)}(s_1,s_2)&=\sum_{(0) \neq I \unlhd \ri }\cczeta_{\G(\ri/I)}(s_1)|\ri:I|^{-s_2},
\end{align*}
respectively, where $s_1$ and $s_2$ are complex variables. 
\end{dfn}
These series converge for $s_1$ and $s_2$ with sufficiently large real parts; see Section~\ref{conv}.

We remark that the zeta functions defined above depend not only on the group of $\ri$-points $\G(\ri)$, but implicitly on the ring~$\ri$ and the
$\ri$-scheme $\G$ which give rise to the finite quotients $\G(\ri/I)$, modulo ideals~$I$ of~$\ri$.  It is convenient and customary to use the short notation $\G(\ri)$ in the index to indicate this more complex dependency.

In Proposition~\ref{factor}, we establish the Euler decompositions
\[\astlvlzf_{\G(\ri)}(s_1,s_2)=\prod_{\p \in \specp} \astlvlzf_{\G(\ric)}(s_1,s_2),\]
where $\ast\in\{\irrup,\ccup\}$ and~$\ric$ is the completion of~$\ri$ at the nonzero prime ideal~$\p$. When considering a fixed prime ideal~$\p$, we write simply $\lri=\ric$
and $\G_N:=\G(\lrip)$. With this notation, the local factor at~$\p$ is given by 
\begin{align}\label{localfactors}\astlvlzf_{\G(\ric)}(s_1,s_2)=\astlvlzf_{\G(\lri)}(s_1,s_2)= \sum_{N=0}^{\infty}\astzeta_{\G_N}(s_1)|\lri:\p|^{-Ns_2}.\end{align}
\begin{exm} Let~$\G(\ri)$ be the free Abelian torsion-free group~$\ri^m$, and let~$\p$ be a nonzero prime ideal of~$\ri$ with $q=|\ri:\p|$. Then, for $N \in \N_0=\N\cup \{0\}$, we have 
\[r_{q^i}(\G_N)=\cc_{q^i}(\G_N)=\begin{cases}
                                                       q^{mN},&\text{ if }i=0,\\
                                                       0,&\text{ otherwise.}
                                                      \end{cases}
\]
Therefore, for $\ast \in \{\irrup,\ccup\}$,
\[\astlvlzf_{\G(\lri)}(s_1,s_2)=\astlvlzf_{\lri^{m}}(s_1,s_2)=\sum_{N=0}^{\infty}q^{N(m-s_2)}=\frac{1}{1-q^{m-s_2}}.\]
Consequently $\astlvlzf_{\ri^m}(s_1,s_2)=\ddkzeta(s_2-m)$ with $\ddkzeta(s)$ denoting the Dedekind zeta function of the number field~$K$.
Moreover, the local factor at~$\p$ satisfies the functional equation
\[\astlvlzf_{\lri^{m}}(s_1,s_2)\mid_{q\to q^{-1}}=-q^{m-s_2}\astlvlzf_{\lri^{m}}(s_1,s_2).\qedhere\]
\end{exm}

Certain zeta functions of groups related to representations---or the local factors of such functions---are known to be rational functions satisfying functional equations; for instance, representation zeta functions of certain pro-$p$ groups~\cite[Theorem~A]{AKOV13}, and local factors of twist representation zeta functions---see Section~\ref{app2}---of groups of the form $\G(\ri)$~\cite[Theorem~A]{StVo14}, where $\G$ is a unipotent group scheme obtained from a nilpotent $\ri$-Lie lattice. 

As for zeta functions related to conjugacy classes, the so-called class number zeta functions---see Section~\ref{app1}---of certain groups are rational; for instance the local factors of class number zeta functions of Chevalley groups $G(\lri)$~\cite[Theorem~C]{BDOP}, where $\lri$ is the valuation ring of a non-Archimedean local field of any (sufficiently large) characteristic, and class number zeta functions of compact $p$-adic analytic groups~\cite[Theorem~1.2]{duSau05}. 

This motivates our main result, which concerns the above mentioned features for the local factors of bivariate representation and conjugacy class zeta functions of groups of the form~$\G(\ri)$ obtained from nilpotent Lie lattices; see Section~\ref{groups}. 

\begin{thm}\label{thmA} Let $\ri$ be the ring of integers of a number field~$K$, and let $\G$ be a unipotent group scheme obtained from a nilpotent $\ri$-Lie lattice~$\Lambda$. For each $\ast \in\{\textup{irr},\textup{cc}\}$, there exist a positive integer $t^{\ast}$ and a rational function 
\[R^{\ast}(X_1, \dots, X_{t^{\ast}}, Y_1,Y_2)\text{ in } \Q(X_1, \dots, X_{t^{\ast}},Y_1,Y_2)\] 
such that, for all but finitely many nonzero prime ideals $\p$ of $\ri$, there exist algebraic integers $\lambda_{1}^{\ast}(\p)$, $\dots$, 
 $\lambda_{t^{\ast}}^{\ast}(\p)$ for which the following holds. For any finite extension $\Lri$ of $\lri:=\ric$ with relative degree of inertia $f=f(\Lri,\lri)$, 
 \[\lvlzf_{\G(\Lri)}^{\ast}(s_1,s_2)=R^{\ast}(\lambda_{1}^{\ast}(\p)^f, \dots, \lambda_{t^{\ast}}^{\ast}(\p)^f,q^{-fs_1},q^{-fs_2}),\]
 where $q=|\ri:\p|$. 
Moreover, inverting parameters in the rational function $R^{\ast}$ yields that these local factors satisfy the functional equation
\[\astlvlzf_{\G(\Lri)}(s_1,s_2)\mid_{\substack{q \rightarrow q^{-1}\\ \lambda_{j}^{\ast}(\p) \rightarrow \lambda_{j}^{\ast}(\p)^{-1}}}=-q^{f(h-s_2)}\astlvlzf_{\G(\Lri)}(s_1,s_2),\] 
 where $h=\dim_K(\Lambda \otimes K)$. 
 \end{thm}
The statement of Theorem~\ref{thmA} is analogous to~\cite[Theorem~A]{StVo14}, and its proof heavily relies on the techniques of~\cite{AKOV13, StVo14}; see Section~\ref{pthmA}. The main tools used in the proof of Theorem~\ref{thmA} are the Kirillov orbit method, the Lazard correspondence, and $\p$-adic integration.

Since the rational functions~$R^{\irrup}$ and $R^{\ccup}$ of Theorem~\ref{thmA} only depend on the $\ri$-scheme $\G$, it follows that almost all local factors $\astlvlzf_{\G(\ric)}(s_1,s_2)$ only depend on~$\G$ and the chosen prime ideal~$\p$ of~$\ri$, not on the group of $\ric$-points $\G(\ric)$. We keep the notation $\astlvlzf_{\G(\ric)}(s_1,s_2)$ to make clear which ring of integers and prime ideal are being considered. 
%

\begin{rmk}\label{rmk:HrMa} 
Local multivariate zeta functions counting the number of equivalence classes in some uniformly definable family of equivalence relations are known to be rational functions; see~\cite[Theorem~1.3]{HrMa}. As an application of this theorem, the authors prove in~\cite[Section~8]{HrMa} rationality for all local factors of twist representation zeta functions of~$\mathcal{T}$-groups, partially extending~\cite[Theorem~A]{StVo14}. Nevertheless, the techniques of~\cite{HrMa} do not assure that these local factors satisfy functional equations, as in~\cite[Theorem~A]{StVo14} or in Theorem~\ref{thmA}. We hope that the methods of~\cite{HrMa} can be applied to the bivariate zeta functions defined here to show rationality of all of their local factors. We point out that if $\G(\ri)$ is a $\mathcal{T}$-group as in Theorem~\ref{thmA} which has nilpotency class~$2$, then all local factors of $\rlvlzf_{\G(\ri)}(s_1,s_2)$ are rational functions; see Remark~\ref{condition2} and Proposition~\ref{intpadic}. 
\end{rmk}

\begin{rmk}
The author is not aware whether the results of Theorem~\ref{thmA} remain true for groups associated to unipotent group schemes in positive characteristic or for arithmetic groups associated to non-unipotent group schemes, since the techniques used here to prove Theorem~\ref{thmA} do not apply in such cases. 
However one could obtain a positive answer for the former question if the techniques of~\cite{HrMa} do apply for these bivariate zeta functions, in which case~\cite[Corollary~6.8]{HrMa} would assure that almost all local factors of the bivariate zeta functions of~$\G(\F_p[t])$ are also rational, where $\F_p$ is the field with~$p$ elements. 
\end{rmk}

Some applications of bivariate zeta functions are given below in Sections~\ref{app1} and~\ref{app2}. Specifically, we obtain results on previously studied (univariate) zeta functions by specializing the bivariate zeta functions introduced here. 

It would be interesting to understand which other kinds of information one can extract from bivariate zeta functions. 
In~\cite{PL18II}, we explicitly compute bivariate zeta functions of three infinite families of nilpotent groups. As a consequence, we obtain explicit formulae for two (univariate) zeta functions of these groups. 
We also provide an application in combinatorics: the formulae for bivariate representation zeta functions of these groups are shown to be related to statistics of certain Weyl groups, leading to formulae for joint distributions of three statistics; see~\cite[Propositions~5.5 and 5.6]{PL18II}. 

\subsection{Application 1: class number zeta functions}\label{app1}
An advantage of the study of the bivariate zeta functions of Definition~\ref{biva} is that they can be used to investigate (univariate) class number zeta functions, which encode the class numbers of principal congruence quotients of the groups considered.
Recall that the \emph{class  number}~$\class(G)$ of a finite group $G$ is the number of its conjugacy classes or, equivalently, the number of its irreducible complex characters. In particular, $\class(G)=\cczeta_G(0)=\rzeta_G(0)$.

\begin{dfn} The \emph{class number zeta function} of the $\mathcal{T}$-group $\G(\ri)$ is
\[\kzeta_{\G(\ri)}(s)=\sum_{(0) \neq I \unlhd \ri}\class(\G(\ri/I))|\ri:I|^{-s},\]
where $s$ is a complex variable.
\end{dfn}

As for the bivariate zeta functions of Definition~\ref{biva}, the class number zeta function defined above depend not only
on the group of $\ri$-points $\G(\ri)$, but also on~$\ri$ and the $\ri$-scheme $\G$. We adopt the notation $\G(\ri)$ in the index to indicate this more complex dependency.

The term `conjugacy class zeta function' is sometimes used for what we call `class number zeta function'; see, for instance,~\cite{BDOP,Ro17,ask2, duSau05}.

Clearly,
\begin{align}\label{specialization}
\rlvlzf_{\G(\ri)}(0,s)=\clvlzf_{\G(\ri)}(0,s)=
\kzeta_{\G(\ri)}(s).
\end{align}

A consequence of Theorem~\ref{thmA} is that almost all local factors of the class number zeta function of $\G(\ri)$ are rational in $\lambda_i(\p)$, $q$, and $q^{-s}$ and 
behave uniformly under base extension. Moreover, for a finite extension $\Lri$ of $\lri$ with relative degree of inertia $f=f(\Lri,\lri)$, the local factors satisfy 
the functional equation 
\[\kzeta_{\G(\Lri)}(s)\mid_{\substack{q \rightarrow q^{-1}\\ \lambda_{j}^{\ast}(\p) \rightarrow \lambda_{j}^{\ast}(\p)^{-1}}}=-q^{f(h-s)}\kzeta_{\G(\Lri)}(s). \]

Rossmann showed independently in~\cite{Ro17} that class number zeta functions of certain nilpotent groups $G \leq \Gl_d(\ric)$ are rational functions and satisfy functional equations. This is a consequence of~\cite[Theorems~4.10 and~4.18]{Ro17} together with
the specialization of ask zeta functions to class number zeta functions given in~\cite[Theorem~1.7]{Ro17}. 

\subsection{Application 2: twist representation zeta functions}\label{app2}
A $\mathcal{T}$-group of nilpotency class~$c=2$ is called a $\mathcal{T}_2$-group. 
The bivariate representation zeta function of a $\mathcal{T}_2$-group $\G(\ri)$ specializes to its twist representation zeta function, whose definition we now recall.

Nontrivial $\mathcal{T}$-groups have infinitely many one-dimensional irreducible complex representations. 
For this reason, one cannot define the representation zeta function of a $\mathcal{T}$-group $G$ as in Definition~\ref{univa}. 
Instead, for a $\mathcal{T}$-group~$G$, one considers equivalence classes on the set of its irreducible complex representations:
two representations $\rho$, $\sigma$ of~$G$ are called~\emph{twist-equivalent} if there exists a one-dimensional 
representation $\chi$ of $G$ such that $\rho \cong \chi \otimes \sigma$. This is an equivalence relation on the set of irreducible complex representations of~$G$ whose equivalence classes are called \emph{twist-isoclasses}. 
Let $\rg_n(G)$ be the number of twist-isoclasses of $n$-dimensional irreducible complex representations of~$G$. If~$G$ is a topological group, we only consider continuous representations. 
The $\rg_n(G)$ are all finite, see~\cite[Theorem~6.6]{LuMa85}. 
\begin{dfn} The \emph{twist representation zeta function} of a $\mathcal{T}$-group $G$ is 
 \[\rtzeta_G(s)=\sum_{n=1}^{\infty}\rg_n(G)n^{-s},\] where $s$ is a complex variable.
\end{dfn}
Twist representation zeta functions of $\mathcal{T}$-groups have been previously investigated, for instance, in~\cite{DuVo14, HrMa, Ro17comp, StVo14, Vo10}.
Explicit examples of (local factors of) twist representation zeta functions of $\mathcal{T}$-groups can be found in~\cite{Ezzphd,Ro17comp,Snophd,StVo14, StVo17}.

Let $\G(\ri)$ be a $\mathcal{T}_2$-group obtained from a unipotent $\ri$-Lie lattice~$\Lambda$ as explained in Section~\ref{groups}.
In Section~\ref{irr}, we show that the bivariate representation zeta function of $\G(\ri)$ specializes to its twist representation zeta function as follows.
For a fixed nonzero prime ideal~$\p$ of~$\ri$, let $\g=\Lambda\otimes_{\ri}\ric$ and let~$\g'$ be the derived Lie sublattice of~$\g$. 
Denote by~$r$ the torsion-free rank of $\g/\g'$. Then Proposition~\ref{repzeta} states
\begin{equation}\label{specializationirr}\prod_{\p \in \specp}\left((1-q^{r-s_2})\rlvlzf_{\G(\ric)}(s_1,s_2)\mid_{\substack{s_1\to s-2\\ s_2\to r\phantom{-2}}}\right)=
\rtzeta_{\G(\ri)}(s),\end{equation}
provided both the left-hand side and the right-hand side converge.

No specialization of the form~\eqref{specializationirr} is expected to exist in case of nilpotency class~$c>2$; see~\cite[Section~3.3]{PLphd} for details. 
In Section~\ref{irr}, we exhibit a $\mathcal{T}$-group of nilpotency class~$3$ whose bivariate representation zeta function does not specialize to its twist representation zeta function. 

We conclude this section with an example which illustrates Theorem~\ref{thmA} and specializations~\eqref{specialization} and~\eqref{specializationirr}. 
\begin{exm}\label{exintro} Let~$\mathbf{H}(\ri)$ denote the Heisenberg group of upper uni-triangular $3\times 3$-matrices over~$\ri$. 
In Example~\ref{Heis}, we show that, for a given nonzero prime ideal~$\p$ of~$\ri$ with $|\ri:\p|=q$, the bivariate zeta functions of~$\mathbf{H}(\lri)$ are given by 
\begin{align}
\label{Heisirr}\rlvlzf_{\mathbf{H}(\lri)}(s_1,s_2)		&=\frac{1-q^{-s_1-s_2}}{(1-q^{1-s_1-s_2})(1-q^{2-s_2})} ~\text{ and}&\\
\label{Heiscc}\clvlzf_{\mathbf{H}(\lri)}(s_1,s_2)		&=\frac{1-q^{-s_1-s_2}}{(1-q^{1-s_2})(1-q^{2-s_1-s_2})}.			&	
\end{align}
In particular, these are rational functions in~$q$, $q^{-s_1}$, and~$q^{-s_2}$, and
\[\astlvlzf_{\mathbf{H}(\lri)}(s_1,s_2)\mid_{q\to q^{-1}}=-q^{3-s_2}\astlvlzf_{\mathbf{H}(\lri)}(s_1,s_2),\]
for each $\ast \in \{ \textup{irr},\textup{cc}\}$.  
Specializations~\eqref{specialization} and~\eqref{specializationirr} yield
\[\kzeta_{\mathbf{H}(\lri)}(s)=\frac{1-q^{-s}}{(1-q^{1-s})(1-q^{2-s})}\hspace{0.3cm}\text{ and }\hspace{0.3cm} \rtzeta_{\mathbf{H}(\lri)}(s)=\frac{1-q^{-s}}{1-q^{1-s}}.\]
The expression of the class number zeta function agrees with the formula given in~\cite[Section~9.3, Table~1]{Ro17}. This example also occurs in~\cite[Section~8.2]{BDOP}, corrected by a sign mistake.
The expression of the twist representation zeta function accords with~\cite[Theorem~B]{StVo14}.
We further note that
\[\kzeta_{\mathbf{H}(\lri)}(s)\mid_{q \to q^{-1}}=-q^{3-s}\kzeta_{\mathbf{H}(\lri)}(s).\qedhere\] 
\end{exm}
%
\section{General properties}
\subsection{Convergence}\label{conv}
It is well known that, if a complex sequence $(a_n)_{n\in\N}$ grows at most polynomially, the Dirichlet series $D((a_n)_{n\in\N},s):=\sum_{n=1}^{\infty}a_nn^{-s}$ converges for 
$s \in \C$ with sufficiently 
large real part. We now show that an analogous result holds for double Dirichlet series.
For simplicity, we write $(a_{n,m}):=(a_{n,m})_{n,m\in\N}$.

\begin{dfn} A double sequence $(a_{n,m})$ of complex numbers is said to have \emph{polynomial growth} if there exist positive integers 
$\alpha_1$ and $\alpha_2$ and a constant $C>0$ such that 
$|a_{n,m}|<Cn^{\alpha_1}m^{\alpha_2}$ for all $n,m \in \N$.
\end{dfn}
\begin{pps}\label{double}
 If the double sequence $(a_{n,m})$ has polynomial growth, then there exist $\alpha_1$, $\alpha_2\in \R$ such that the double Dirichlet series
 \[D((a_{n,m}), s_1, s_2):=\sum_{n=1}^{\infty}\sum_{m=1}^{\infty}a_{n,m}n^{-s_1}m^{-s_2}\]
 converges absolutely for $(s_1,s_2)\in \C^2$ satisfying $\re(s_1)>\alpha_1$ and $\re(s_2)>\alpha_2$. 
\end{pps}
\begin{proof}
Let $\beta_1$, $\beta_2 \in \N$ and $C>0$ be such that $|a_{n,m}|<Cn^{\beta_1}m^{\beta_2}$, for all $n,m\in \N$. Then
\[\sum_{n=1}^{\infty}\sum_{m=1}^{\infty} \left|\frac{a_{n,m}}{n^{s_1}m^{s_2}}\right|\leq C\sum_{n=1}^{\infty}\sum_{m=1}^{\infty} \frac{1}{n^{\re(s_1)-\beta_1}m^{\re(s_2)-\beta_2}}.\]
The relevant statement of Proposition~\ref{double} then follows from the fact that, for $p,q \in \R$, the harmonic double series 
\[\sum_{k=1}^{\infty}\sum_{l=1}^{\infty}\frac{1}{k^{p}l^{q}}\] 
converges if and only if $p>1$ and $q>1$; see \cite[Example~7.10(iii)]{GhLi}.
\end{proof}

For a unipotent $\ri$-group scheme~$\G$ and positive integers~$m$ and $n$, write
\[r_{n,m}(\G(\ri))=\sum_{\substack{I \unlhd \ri\\ |\ri:I|=m}}r_n(\G(\ri/I)) 
\text{ and }c_{n,m}(\G(\ri))=\sum_{\substack{I \unlhd \ri\\ |\ri:I|=m}}c_n(\G(\ri/I)).\]
The bivariate representation and the bivariate conjugacy class zeta functions of $\G(\ri)$ are given by the following double Dirichlet series with nonnegative coefficients:
\begin{align*}
 \rlvlzf_{\G(\ri)}(s_1,s_2)=\sum_{n=1}^{\infty}\sum_{m=1}^{\infty}r_{n,m}(\G(\ri))n^{-s_1}m^{-s_2},\\
 \clvlzf_{\G(\ri)}(s_1,s_2)=\sum_{n=1}^{\infty}\sum_{m=1}^{\infty}c_{n,m}(\G(\ri))n^{-s_1}m^{-s_2}.
\end{align*}

\begin{pps}
The bivariate zeta functions $\rlvlzf_{\G(\ri)}(s_1,s_2)$ and $\clvlzf_{\G(\ri)}(s_1,s_2)$ converge (at least) on some open domain of the form 
\begin{equation}\label{DAB}\mathscr{D}_{\alpha_1,\alpha_2}:=\{(s_1,s_2)\in \C^2\mid \re(s_1)>\alpha_1,~\re(s_2)> \alpha_2\},
\end{equation}
for some real constants $\alpha_1$ and $\alpha_2$.
\end{pps}
\begin{proof}
Set $\gamma_m(\ri):=|\{I  \unlhd \ri \mid |\ri:I|=m\}|$.
The Dedekind zeta function of the number field $K$ is given by $\ddkzeta(s)=\sum_{m=1}^{\infty}\gamma_m(\ri)m^{-s}$,
and is known to converge for $\re(s)>1$. In particular, there exists a positive real constant~$C$ such that, for each $M \in \N$, it holds that $\sum_{m=1}^{M}\gamma_m(\ri)<CM$.

Given $I \unlhd \ri$, the finite group $\G(\ri/I)$ is a congruence quotient of a torsion-free nilpotent and finitely generated group. 
Then there exists $\mathscr{P}(X)\in \Z[X]$ such that, for all $I \unlhd \ri$, the cardinality of $\G(\ri/I)$ is bounded by $\mathscr{P}(m)$, where $m=|\ri:I|$. 

Given $I \unlhd \ri$, the finite group $\G(\ri/I)$ has at most $|\G(\ri/I)|$ conjugacy classes.
Consequently, for each $(n,m)\in \N^2$, 
\[c_{n,m}(\G(\ri))=\sum_{\substack{ I \unlhd \ri \\ |\ri:I|=m}}c_n(\G(\ri/I)) < C m\mathscr{P}(m).\]
Analogously, $r_{n,m}(\G(\ri)) < C m\mathscr{P}(m)$, since $r_n(\G(\ri/I)) \leq |\G(\ri/I)|$.
\end{proof}

When finite, the abscissa of convergence of a Dirichlet series $\sum_{n=1}^{\infty}a_nn^{-s}$ gives the precise degree of polynomial growth of the sequence $(\sum_{i=1}^{n}a_i)_{n\in \N}$. 
However, for double Dirichlet series $\sum_{m=1}^{\infty}\sum_{n=1}^{\infty}a_{n,m}n^{-s_1}m^{-s_2}$, an analogue of an abscissa of convergence might not be unique. 
As mentioned in Example~\ref{exintro}, the bivariate representation zeta function of the Heisenberg group $\mathbf{H}(\ri)$ is given by
\begin{align*} \rlvlzf_{\mathbf{H}(\lri)}(s_1,s_2)	 =\frac{1-q^{-s_1-s_2}}{(1-q^{1-s_1-s_2})(1-q^{2-s_2})}.\end{align*} 
The maximal domain of convergence of $\rlvlzf_{\mathbf{H}(\lri)}(s_1,s_2)$ is
\[\mathscr{D}_{\mathbf{H}}:=\{ (s_1,s_2)\in \C^2  \mid \re(s_1+s_2)> 2 \text{ and }\re(s_2)>3\}.\]
In contrast with the one variable case, there is more than one choice of constants $(\alpha_1,\alpha_2) \in \R^2$ such that~$\mathscr{D}_{\alpha_1,\alpha_2}$ is a maximal domain of the form~\eqref{DAB} with the property that $\rlvlzf_{\mathbf{H}(\ri)}(s_1,s_2)$ converges on it. 
For instance, $\mathscr{D}_{-1,3}$ and $\mathscr{D}_{-2,4}$ are two such domains. 
However, no choice of $(\alpha_1,\alpha_2)$ is such that~$\mathscr{D}_{\alpha_1,\alpha_2}$ coincides with the maximal domain of convergence~$\mathscr{D}_{\mathbf{H}}$ of $\rlvlzf_{\mathbf{H}(\lri)}(s_1,s_2)$. 

\subsection{Euler products}
Our main results concern properties of local factors of bivariate representation and bivariate conjugacy class zeta functions.
In this section, we show that the corresponding global zeta functions can be written as products of such local terms, allowing us to relate local results to the global zeta 
functions. Here, $\G$~denotes a unipotent $\ri$-group scheme. 
\begin{pps}\label{factor} 
For each $\ast \in \{\textup{irr},\textup{cc}\}$ and for $s_1$ and $s_2$ with sufficiently large real parts, the following Euler decomposition holds.
\[ \astlvlzf_{\G(\ri)}(s_1,s_2)=
\prod_{\p \in \specp} \astlvlzf_{\G(\ric)}(s_1,s_2).\]
\end{pps}
\begin{proof} 
It suffices to show that, for any ideal $I$ of finite index in $\ri$ with prime decomposition $I=\p_{1}^{e_1}\cdots \p_{r}^{e_r}$, with $\p_i \neq \p_j$ if $i \neq j$, 
the following holds
\[\astzeta_{\G(\ri/I)}(s)=\prod_{i=1}^{r}\astzeta_{\G(\ri/\p_{i}^{e_i})}(s).\]
Unipotent groups satisfy the strong approximation property; see~\cite[Lemma~5.5]{PlRa94}. This gives an isomorphism 
\begin{equation}\label{sap}\G(\ri/I) \cong \G(\ri/\p_{1}^{e_1})\times \dots \times \G(\ri/\p_{r}^{e_r}).\end{equation}
We first show the relevant statement of Proposition~\ref{factor} for the representation case. 
Given a positive integer~$n$, write $[n]=\{1,\dots, n\}$.
For a group~$G$, denote by $\Irr(G)$ the set of its irreducible characters. A consequence of~\eqref{sap} is  
\[\Irr(\G(\ri/I)) \cong \Irr(\G(\ri/\p_{1}^{e_1}))\times \dots \times \Irr(\G(\ri/\p_{r}^{e_r})).\]
Write $\Irr_i=\Irr(\G(\ri/\p_{i}^{e_i}))$. 
Since $r_n(\G(\ri/I))=|\{\chi \in \Irr(\G(\ri/I)): \chi(1)=n\}|$, it follows that
\begin{align*}
\rzeta_{\G(\ri/I)}(s)=\sum_{\chi \in \Irr(\G(\ri/I))}\chi(1)^{-s} 	&=\sum_{(\chi_1, \dots, \chi_r) \in \Irr_1\times \dots \times \Irr_{r}} \chi_1(1)^{-s}\cdots\chi_r(1)^{-s}\\ 
=\prod_{i=1}^{r}\sum_{\chi_i \in \Irr_i}\chi_i(1)^{-s} \hspace{0.35cm}
			&=\prod_{i=1}^{r}\rzeta_{\G(\ri/\p_{i}^{e_i})}(s).
\end{align*}

For the conjugacy class zeta function, we use the fact that each conjugacy class~$C$ of $\G(\ri/\p_{1}^{e_1})\times \dots \times \G(\ri/\p_{r}^{e_r})$ is of the form 
$C=C_1 \times \dots \times C_r$, where~$C_i$ is a conjugacy class of $\G(\ri/\p_{i}^{e_i})$, for each $i\in [r]$. Thus
\[\cc_n(\G(\ri/I))=\sum_{\substack{n_1,\dots,n_r\in \N_0\\n_1 \cdots n_r=n}}\cc_{n_1}(\G(\ri/\p_{1}^{e_1}))\cdots\cc_{n_r}(\G(\ri/\p_{r}^{e_r})).\]
Set $q_i=|\ri:\p_i|$. In Section~\ref{chcc}, we show that all conjugacy classes of $\G(\ri/\p_{i}^{e_i})$ have size a power of~$q_i$. Consequently 
\begin{align*}
 \cczeta_{\G(\ri/I)}(s)	&=\sum_{n=1}^{\infty}\sum_{\substack{n_1,\dots,n_r\in \N_0\\q_{1}^{n_1} \dots q_{r}^{n_r}=n}}\cc_{q_{1}^{n_1}}(\G(\ri/\p_{1}^{e_1})) \cdots 
			      \cc_{q_{r}^{n_r}}(\G(\ri/\p_{r}^{e_r}))(q_{1}^{n_1}\dots q_{r}^{n_r})^{-s}&\\ 
			&=\prod_{k=1}^{r}\left(\sum_{n_k=0}^{\infty}\cc_{q_{k}^{n_k}}(\G(\ri/\p_{k}^{e_k}))q_{k}^{-n_ks}\right)=
			      \prod_{k=1}^{r}\cczeta_{\G(\ri/\p_{k}^{e_k})}(s).\qedhere
 \end{align*}		
 \end{proof}

\section{Group schemes obtained from nilpotent Lie lattices}\label{groups}
Here, we explain briefly how the unipotent groups schemes $\G$ of interest in this work are constructed; for more details see~\cite[Section~2.1.2]{StVo14}.

An $\ri$-Lie lattice is a free and finitely generated $\ri$-module~$\Lambda$ together with an antisymmetric $\ri$-bilinear form $[\cdot,\cdot]$ which satisfies the Jacobi identity. 

Let $(\Lambda, [\cdot,\cdot])$ be a nilpotent $\ri$-Lie lattice of $\ri$-rank~$h$. Let $\mathscr{B}=(x_1, \dots, x_h)$ be an $\ri$-basis for~$\Lambda$. 
For each $\ri$-algebra $R$, set $\Lambda(R)=\Lambda \otimes_{\ri} R$. Then 
\[\mathscr{B}_R=(x_1 \otimes 1_R, \dots, x_h\otimes 1_R)\] is an $R$-basis of~$\Lambda(R)$. 

If~$\Lambda$ has nilpotency class~$c$ and satisfies $\Lambda' \subseteq c!\Lambda$, where $\Lambda'=[\Lambda,\Lambda]$ is its derived Lie sublattice, we may define a group operation~$\ast$ on~$\Lambda(\ri)$ by means of the Hausdorff series. 
The assumption $\Lambda' \subseteq c!\Lambda$ assures that all denominators in the Hausdorff series will cancel out, so that the coordinates of the operation~$\ast$ in terms of~$\mathscr{B}$ are given by polynomials $f_1$, $\dots$, $f_h: \ri^{2h} \to \ri$, say, with coefficients in~$\ri$. That is, given $a=\sum_{i=1}^{h}a_ix_i$, $b=\sum_{i=1}^{h}b_ix_i \in \Lambda(\ri)$, we have
\[a \ast b =\sum_{i=1}^{h}f_i(\underline{a},\underline{b})x_i,\] 
where $\underline{a}=(a_1, \dots, a_h)$, $\underline{b}=(b_1, \dots, b_h) \in \ri^h$.

For each $\ri$-algebra~$R$, one may give $\Lambda(R)$ a group structure by defining the group operation~$\ast$ on $\Lambda(R)$ as follows.
\[\text{ for all } a,b \in \Lambda(R): a \ast b = \sum_{i=1}^{f}f_i(\underline{a},\underline{b}) (x_i\otimes 1_R),\]
where \[a=\sum_{i=1}^{h}a_i(x_i\otimes 1_R),~ b=\sum_{i=1}^{h}b_i(x_i\otimes 1_R) \in \Lambda(R),\] 
and $\underline{a}=(a_1, \dots, a_h)$, $\underline{b}=(b_1, \dots, b_h) \in \ri^h$.
This defines a unipotent $\ri$-group scheme $\G=\G_{\Lambda}$ which is isomorphic as a scheme to affine $h$-space over~$\ri$ which represents the group functor $R \mapsto (\Lambda(R),\ast)$.

The $\ri$-group scheme~$\G$ is such that $\G(\ri)$ is a $\mathcal{T}$-group of nilpotency class~$c$. 
If~$R$ is an $\ri$-algebra whose underlying additive group is a finitely generated pro-$p$ group, for instance $R=\ric$, then $\G(R)$ is a finitely generated pro-$p$ group of nilpotency class~$c$. 

For Lie lattices~$\Lambda$ of nilpotency class~$2$, a different construction of such unipotent group schemes is given in~\cite[Section~2.4.1]{StVo14}, in which case the hypothesis $\Lambda' \subseteq 2\Lambda$ is not needed. However, if this condition is satisfied, the unipotent group schemes obtained via such construction coincide with the latter ones.

\begin{rmk}
Definition~\ref{biva} of bivariate zeta functions of groups~$\G(\ri)$ might be extended to $\mathcal{T}$-groups which are not necessarily of the form~$\G(\ri)$. 

As explained in~\cite[Section~5]{DuVo14} every $\mathcal{T}$-group~$G$ is virtually of the form~$\G(\Z)$ for some unipotent $\Z$-group scheme~$\G$ obtained from a nilpotent $\Z$-Lie lattice. We might then define $\astlvlzf_G(s_1,s_2)$ as the corresponding bivariate zeta function of a finite index subgroup~$H$ of~$G$ which is of the form~$\G(\Z)$.
%
%
That is,
\[\astlvlzf_{G}(s_1,s_2)=\astlvlzf_{G,H}(s_1,s_2):=\astlvlzf_{\G(\Z)}(s_1,s_2),~\ast\in\{\irrup, \ccup\}.\] 

If~$G$ has two subgroups $H_1=\G_1(\Z)$ and $H_2=\G_2(\Z)$ of finite index, then $H_1$ and $H_2$ are commensurable and, therefore, they have the same pro-$p$ completion for all but finitely many prime integers~$p$; see~\cite[Lemma~1.8]{Pi71}. 
In particular, $\astlvlzf_{\G_1(\Z_p)}(s_1,s_2)=\astlvlzf_{\G_2(\Z_p)}(s_1,s_2)$, for all but finitely many primes~$p$, that is, although $\astlvlzf_{G,H_1}(s_1,s_2)$ and $\astlvlzf_{G,H_2}(s_1,s_2)$ may not coincide, they are almost the same in the sense that they coincide except for finitely many local factors.

In this point of view, bivariate zeta functions also yield invariants for $\mathcal{T}$-groups, namely domains of convergence and of meromorphy.
More precisely, we prove in~\cite[Theorem~1]{PL18III} that the maximal domains of convergence of the bivariate zeta functions of groups of the form $\G(\ri)$---when finitely many local factors are disregarded---are independent of the ring of integers~$\ri$ and admit meromorphic continuations to domains which are also independent of~$\ri$. 
\end{rmk}

From now on, we assume that~$\G$ is a unipotent $\ri$-group scheme obtained from a nilpotent $\ri$-Lie lattice~$\Lambda$.

\section{Bivariate zeta functions and \texorpdfstring{$\p$}{p}-adic integrals}\label{padicintegral}
Our results rely on the fact that local bivariate representation and local bivariate conjugacy class zeta functions of groups associated to unipotent group schemes can be written in terms of $\p$-adic integrals. 

In Section~\ref{spadic}, we show how to write most of the local factors of these zeta functions in terms of $\p$-adic integrals using the methods of~\cite[Section~2.2]{Vo10}. A consequence is that these local factors are given by  rational functions as stated in Theorem~\ref{thmA}.
In Section~\ref{pthmA}, we use the results of~\cite[Section~2.1]{Vo10} and~\cite[Section~4.1]{AKOV13} to show that these local factors satisfy functional equations and that they are uniform under base extension, concluding the proof of Theorem~\ref{thmA}.

The methods of~\cite[Section~2]{Vo10} were also applied in~\cite{Ro17} to show functional equations for the ask zeta functions of modules of matrices over compact discrete valuation rings~$\Lri$ in characteristic zero. This result provides functional equations for the class number zeta functions of certain nilpotent groups $G \leq \Gl_d(\Lri)$.

Furthermore, these methods were applied to show results similar to Theorem~1.4 in~\cite{AKOV13}, for the representation zeta functions of certain $p$-adic analytic groups, and in~\cite{StVo14}, for the local factors of twist representation zeta functions of groups of the form~$\G(\ri)$. 

We now recall the methods of~\cite[Section~2.2]{Vo10}. 

For the rest of this section, let $\p$~be a fixed nonzero prime ideal of~$\ri$ and $\lri=\ric$. 
Denote by~$q$ the cardinality of~$\ri/\p$ and by~$p$ its characteristic.

Recall that given an element $z \in \lri$, the ideal $(z) \lhd \ri$ has prime factorisation $(z)=\p^{e}\p_{1}^{e_1}\cdots\p_{r}^{e_r}$ such that $\p_i\neq \p$, 
for all $i\in[r]$. The $\p$-\emph{adic valuation} of $z$ is $v_{\p}(z)=e$, and its $\p$-\emph{adic norm} is $|z|_{\p}=q^{-v_{\p}(z)}$. Equivalently,  $v_{\p}(z)=e$ and $|z|_{\p}=q^{-e}$ if $z \in \p^e\setminus \p^{e+1}$.

For each $j \in \N$, denote by $\|\cdot\|_{\p}$ the maximum norm of $\lri^j$ with respect to $|\cdot|_{\p}$; that is, for $\mathbf{z}=(z_1, \dots , z_j)\in \lri^j$, let $\|\mathbf{z}\|_{\p}=\max\{|z_k|_{\p}\}_{k=1}^{j}$. 

For $N \in \N$, we also denote by $v_{\p}$ the function on $\lri/\p^N$ given as follows: let~$\overline{z}$ be the image of $z \in \lri$ under $\lri \to \lri/\p^N$ and assume $z \in \p^e \setminus \p^{e+1}$. Then $v_{\p}(\overline{z})=e$ if $0 \leq e <N$, and $v_{\p}(\overline{z})=+\infty$, otherwise.

We make the following distinction: $\p^m$ denotes the $m$th ideal power $\p\cdots\p$, whilst $\p^{(m)}$ denotes the $m$-fold Cartesian power $\p \times \dots \times \p$.

For $k$, $ N\in \N$, set
\begin{align*}W_{k}(\lrip)&:=((\lrip)^k)^{*}=\{\mathbf{x}\in(\lrip)^k\mid v_{\p}(\mathbf{x})=0\},\\
\W_{k}&:=(\lri^k)^{*}\phantom{(/p^N)}=\{\mathbf{x}\in\lri^k\mid v_{\p}(\mathbf{x})=0\},
\end{align*}
and let $W_{k}((0))=(0)^k$ for each $k \in \N$. 

Let $\pi\in \lri$ be a uniformizer of $\lri$. Given a matrix $M\in \Mat_{m \times n}(\lri/\p^N)$, we write $\nu(M)=(m_1, \dots, m_{\epsilon})$ to indicate the elementary divisor type of~$M$, where $0 \leq \epsilon \leq \min\{m,n\}$. 

In the following, denote by~$\text{Frac}(\lri)$ the field of fractions of~$\lri$.
Let $n \in \N$ and let $\mathcal{R}(\underline{Y})=\mathcal{R}(Y_1, \dots, Y_n)$ be a matrix of polynomials $\mathcal{R}(\underline{Y})_{ij}\in\lri[\underline{Y}]$ with
\[u_\mathcal{R}=\max\{\rk_{\text{Frac}(\lri)}\mathcal{R}(\mathbf{z}) \mid \mathbf{z}\in \lri^n\}.\] For each $\mathbf{m}\in \N_{0}^{u_\mathcal{R}}$, write
 \begin{align*}
\mathfrak{N}_{\m}^{\mathcal{R}}(\lrip)&:=\{\mathbf{y} \in W_{n}(\lrip) \mid \vry= \mathbf{m}\}\text{ and }\\
\Nr_{\m}(\lrip)&:=|\mathfrak{N}_{\m}^{\mathcal{R}}(\lrip)|.\nonumber 
 \end{align*} 
The number $\Nr_{\m}(\lrip)$ is zero unless $\mathbf{m}=(m_1, \dots, m_{\ur})$ satisfies \[0 = m_1 \leq \dots \leq m_{\ur} \leq N.\] 

Let  $\underline{r}=(r_1, \dots, r_{\ur})$ be a vector of variables. Consider the Poincaré series
\begin{equation}\label{poincare}\Poin_{\mathcal{R}}(\underline{r},t)=\sum_{\substack{N \in \N\\ \m \in \N_{0}^{\ur}}}
\Nr_{\m}(\lrip)q^{-tN-\sum\limits_{i=1}^{\ur}r_im_i}.
\end{equation}

In~\cite[Section 2.2]{Vo10} it is shown how to describe the series~\eqref{poincare} in terms of the $\p$-adic integral
\begin{equation}\label{Vo10}\pint_{\mathcal{R}}(\underline{r},t)=\frac{1}{1-q^{-1}}\int_{(x,\underline{y}) \in \p\times \W_{n}}|x|_{\p}^{t}\prod_{k=1}^{\ur}
\frac{\|F_{k}(\mathcal{R}(\underline{y})) \cup xF_{k-1}(\mathcal{R}(\underline{y}))\|_{\p}^{r_k}}{\|F_{k-1}(\mathcal{R}(\underline{y}))\|_{\p}^{r_{k}}}d\mu,
\end{equation}
where $\mu$ is the additive Haar measure normalised so that $\mu(\lri^{n+1})=1$; 
$F_j(\mathcal{R}(\underline{y}))$ is the set of nonzero $j \times j$-minors of $\mathcal{R}(\underline{y})$.

More precisely, in~\cite[Section 2.2]{Vo10} it is shown that~\eqref{poincare} satisfies
\begin{equation}\label{Vo14gen}
\Poin_{\mathcal{R}}(\underline{r},t)=\pint_{\mathcal{R}}(\underline{r},t-n-1). 
\end{equation}

Suppose now that $M\in \Mat_{n\times n}(\lri/\p^N)$ is an antisymmetric matrix. Then, for some $\xi \in [n]_0:=\{0,1,\dots,n\}$, the elementary divisor type of~$M$ is of the form
\[\nu(M)=(m_1, m_1, m_2, m_2, \dots, m_{\xi},m_{\xi}).\]
If $M$ is antisymmetric, we write $\tilde{\nu}(M)=(m_1, m_2, \dots, m_{\xi})$ for its reduced elementary divisor type, that is, to indicate~$\nu(M)=(m_1, m_1, m_2, m_2, \dots, m_{\xi},m_{\xi})$.

Assume now that $\mathcal{R}(\underline{Y})$ is antisymmetric, in which case $\ur$ is even. For each $\m \in \N_{0}^{\ur/2}$, write 
\begin{align*}\widetilde{\mathfrak{N}}_{\m}^{\mathcal{R}}(\lrip)&=\{\mathbf{y} \in W_{n}(\lrip) \mid \tilde{\nu}(\mathcal{R}(\textbf{y}))= \mathbf{m}\} \text{ and } \\
\Nr_{\m}(\lrip)&=|\widetilde{\mathfrak{N}}_{\m}^{\mathcal{R}}(\lrip)|.
\end{align*}

For $\mathcal{R}(\underline{Y})$ antisymmetric, we assume that the vector of variables $\underline{r}$ is of the form $\underline{r}=\left(r_1,r_1,\dots,r_{\ur/2},r_{\ur/2}\right)$ so that \vspace{-0.3cm}
\begin{equation}\label{poincareas}\Poin_{\mathcal{R}}(\underline{r},t)=\sum_{N \in \N,~\m \in \N_{0}^{\ur/2}}
\Nr_{\m}(\lrip)q^{-tN-2\sum\limits_{i=1}^{\ur/2}r_im_i}.
\end{equation}
Recall the notation $[n]=\{1,\dots, n\}$, for $n \in \N$. Given $x\in \lri$ with $v_{\p}(x)=N$, $\mathbf{y}\in\lri^{n}$, and $k\in[\ur]$, we obtain from \cite[Lemma~4.6(i) and~(ii)]{Ro17} the following for the antisymmetric 
matrix $\mathcal{R}(\mathbf{y})$ with $\tilde{\nu}(\mathcal{R}(\mathbf{y}))=(m_1, \dots, m_{\ur})$:
\[\frac{\|F_{2k}(\mathcal{R}(\mathbf{y})) \cup x F_{2k-1}(\mathcal{R}(\mathbf{y}))\|_{\p}}{\|F_{2k-1}(\mathcal{R}(\mathbf{y}))\|_{\p}}=
\frac{\|F_{2k-1}(\mathcal{R}(\mathbf{y})) \cup x F_{2(k-1)}(\mathcal{R}(\mathbf{y}))\|_{\p}}{\|F_{2(k-1)}(\mathcal{R}(\mathbf{y}))\|_{\p}}=q^{-\min(m_k,N)},\] 
and 
\[\frac{\|F_{2k}(\mathcal{R}(\mathbf{y})) \cup x^2 F_{2(k-1)}(\mathcal{R}(\mathbf{y}))\|_{\p}}{\|F_{2(k-1)}(\mathcal{R}(\mathbf{y}))\|_{\p}}=q^{-2\min(m_k,N)}.\]

Therefore, if $\mathcal{R}(\underline{Y})$ is an antisymmetric matrix, the series~\eqref{poincareas} can be described by the $\p$-adic integral
\begin{align}&\Poin_{\mathcal{R}}(\underline{r},t)=
\pint_{\mathcal{R}}(\underline{r},t-n-1)=\nonumber\\ 
&\label{Vo10as}\frac{1}{1-q^{-1}}\int_{(x,\underline{y}) \in \p\times \W_{n}}|x|_{\p}^{t-n-1}
\prod_{k=1}^{\ur/2}\frac{\|F_{2k}(\mathcal{R}(\underline{y})) \cup x^2F_{2(k-1)}(\mathcal{R}(\underline{y}))\|_{\p}^{r_k}}
{\|F_{2(k-1)}(\mathcal{R}(\underline{y}))\|_{\p}^{r_{k}}}d\mu.
\end{align}

\subsection{The numbers \texorpdfstring{$r_n(\G_N)$}{rn(GN)} and \texorpdfstring{$\cc_n(\G_N)$}{cn(GN)}}\label{chcc}
Recall the notation $\G_N=\G(\lrip)$. We now write the local bivariate zeta functions at $\p$ in terms of sums encoding the elementary divisor types of certain matrices 
associated to $\Lambda$. 
This is done by rewriting the numbers $r_n(\G_N)$ and $c_n(\G_N)$, for $n\in \N$ and $N \in \N_0$, in terms of numbers~$\Nr_{\m}(\lrip)$ defined at the beginning of Section~\ref{padicintegral}. In each case, $\mathcal{R}$ is one of the two commutator matrices of 
$\Lambda$ which we now define.

Set $\g=\Lambda(\lri)=\Lambda\otimes_{\ri} \lri$. Let $\g'$ be the derived Lie sublattice of $\g$, and let $\z$ be its center. Consider the torsion-free $\ri$-ranks  
\begin{align*} & h=\rk(\g), & & a=\rk(\g/\z), & & b=\rk(\g'), & & r= \rk(\g/\g'), & & z=\rk(\z) .\end{align*} 
For~$R$ either~$\ri$ or~$\lri$, let~$M$ be a finitely generated $R$-module with a submodule~$N$. The \emph{isolator}~$\iota(N)$ of~$N$ in~$M$ is the smallest submodule~$L$ of~$M$ containing~$N$ such that~$M/L$ is torsion free. In particular $\z=\iota(\z)$; see~\cite[Lemma~2.5]{StVo14}. 
Set $k=\rk(\iota(\g')/\iota(\g'\cap\z))=\rk(\iota(\g'+\z)/\z)$.

The commutator matrices are defined with respect to a fixed $\lri$-basis 
\[\mathscr{B}=(e_1, \dots, e_h)\] of the $\lri$-Lie lattice $\g$, satisfying the conditions
\begin{align*}
 (e_{a-k+1}, \dots, e_{a}) 	& \text{ is an $\lri$-basis for }\iota(\g'+\z),\\
 (e_{a+1}, \dots, e_{a-k+b})	& \text{ is an $\lri$-basis for }\iota(\g' \cap \z),\text{ and}\\
 (e_{a+1}, \dots, e_{h})	&\text{ is an $\lri$-basis for }\z.
\end{align*}
Denote by $\overline{\phantom{X}}$ the natural surjection $\g \to \g/\z$. 
Let $\mathbf{e}=(e_1, \dots, e_a)$. Then 
\[\overline{\mathbf{e}}=(\overline{e_1}, \dots, \overline{e_a})\] is an $\lri$-basis of $\g/\z$.
The $e_i$ can be chosen so that there are nonnegative integers $c_1,\dots, c_{b}$ with the property that 
\begin{align*}
 (\overline{\pi^{c_1}e_{a-k+1}}, \dots, \overline{\pi^{c_k}e_{a}})	&\text{ is an $\lri$-basis of }\overline{\g'+\z}\text{ and}\\
 (\pi^{c_{k+1}}e_{a+1}, \dots, \pi^{c_b}e_{a-k+b})			&\text{ is an $\lri$-basis of }\g'\cap\z,
\end{align*}
by the elementary divisor theorem.
Fix an $\lri$-basis $\mathbf{f}=(f_1, \dots, f_b)$ for $\g'$ satisfying
\begin{align*}
 (\overline{f_1}, \dots, \overline{f_k})&=(\overline{\pi^{c_1}e_{a-k+1}}, \dots, \overline{\pi^{c_k}e_{a}})	\text{ is an $\lri$-basis of }\overline{\g'+\z}\text{ and}\\
 (f_{k+1}, \dots, f_{b})&=(\pi^{c_{k+1}}e_{a+1}, \dots, \pi^{c_b}e_{a-k+b})					\text{ is an $\lri$-basis of }\g'\cap\z.
\end{align*}
For $i,j \in [a]$ and $k \in [b]$, let $\lambda_{ij}^{k}\in \lri$ be the structure constants satisfying  \vspace{-0.2cm}
\[[e_i,e_j]=\sum_{k=1}^{b}\lambda_{ij}^{k}f_k.\] \vspace{-0.2cm}
The following matrices were previously defined in~\cite[Definition~2.1]{ObVo15}.
\begin{dfn}\label{commutator} The $A$\emph{-commutator} and the $B$\emph{-commutator matrices} of $\g$ with respect to $\mathbf{e}$ and $\mathbf{f}$ are
\begin{align*}
 A(X_1, \dots, X_a)&=\left( \sum_{j=1}^{a}\lambda_{ij}^{k}X_j\right)_{ik}\in \Mat_{a \times b}(\lri[\underline{X}]),\text{ and} \\
 B(Y_1,\dots, Y_b)&=\left( \sum_{k=1}^{b}\lambda_{ij}^{k}Y_k\right)_{ij}\in \Mat_{a \times a}(\lri[\underline{Y}]),
\end{align*}
respectively, where $\underline{X}=(X_1, \dots, X_a)$ and $\underline{Y}=(Y_1, \dots, Y_b)$ are independent variables. 
\end{dfn}
For each $\mathbf{y}\in \lri^b$, the matrix $B(\mathbf{y})$ is antisymmetric.
Fix $N \in \N$. The congruence quotient $\G_N$ is a finite $p$-group of nilpotency class~$c$. 
Set $\gn:=\Lambda \otimes_{\lri} \lrip$ and $\zn=\z\otimes_{\lri} \lrip$, and let $\gnc=\g\otimes_\lri \lrip$.

Tensoring~$\mathbf{e}$ and $\mathbf{f}$ with $\lrip$ yields ordered sets 
\[\mathbf{e}_N=(e_{1}^{N}, \dots, e_{a}^{N})\text{ and }\mathbf{f}_N=(f_{1}^{N},\dots,f_{b}^{N})\] such that 
$\overline{\mathbf{e}}=(\overline{e_{1}^{N}}, \dots,\overline{e_{a}^{N}})$ is an $\lrip$-basis for $\gn/\zn$ and $\mathbf{f}_N$ is an $\lrip$-basis for $\gnc$ as $\lrip$-modules, where $\overline{~\cdot~}$ is the natural surjection $\gn \to \gn/\zn$.

Given an element $\omega$ of $\widehat{\gn}=\Hom_{\lri}(\gn,\C^{\times})$, set 
\[ B_{\omega}^{N}:  \gn \times \gn ~\rightarrow 	~	 \C^{\times}  ,~ (u,v) 		 \mapsto	 \omega([u,v]).\] 
The \emph{radical} of $B_{\omega}^{N}$ is 
\[\rad=\{u \in \gn\mid \text{ for all } v \in \gn :~B_{\omega}^{N}(u,v)=1 \}.\]
Observe that $B_{\omega}$ depends only on the restriction of~$\omega$ to $\gnc$. 
For this reason, given $\widetilde{\omega} \in \widehat{\gnc}$ and any extension $\omega$ of~$\widetilde{\omega}$ to $\gn$, we write~$B_{\widetilde{\omega}}$ for~$B_{\omega}$.

For $x \in \gn/\zn$, following~\cite[Section~3.1]{ObVo15}, we define 
\begin{alignat*}{5} \ad_x:\gn/\zn & \to \gnc 	&\hspace{0.5cm}&\text{and}\hspace{0.5cm} & \ad_{x}^{\star}:  \widehat{\gnc} &\to \widehat{\gn/\zn}\\
					y & \mapsto [y,x] 			&\hspace{0.5cm}&& ~\omega 	& \mapsto \omega \circ \ad_{x}. 
\end{alignat*}
Observe that in the definition of the $\lrip$-module homomorphism $\ad_x$ we identify~$\gn/\zn$ with the $\lrip$-submodule of~$\gn$ generated by~$\textbf{e}_N$. 

The dimensions of the irreducible complex representations and the sizes of conjugacy classes of~$\G_N$ are powers of~$p$ and, according to~\cite[Section~3]{ObVo15}, 
for $c<p$, the numbers $r_{p^i}(\G_N)$ and $c_{p^i}(\G_N)$ are given by
\begin{align}
 \label{chi} r_{p^i}(\G_N)&=\left|\left\{\omega \in \widehat{\gnc} \bigm| |\rad:\zn|=p^{-2i}|\gn/\zn|\right\}\right|\hspace{0.05cm}|\gn/\gnc|~p^{-2i}, \\[0.1cm] 
 \label{cci}\cc_{p^i}(\G_N)&=\left|\left\{x \in \gn/\zn \bigm| |\Ker(\ad_{x}^{\star})|=p^{-i}|\widehat{\gnc}|\right\}\right|\hspace{0.05cm}|\zn|~p^{-i}.
\end{align}
The first formula is a consequence of the Kirillov orbit method, which reduces the problem of enumerating the characters of $\G_N$ to the problem of determining the indices 
in $\gn$ of $\rad$ for $\omega \in \widehat{\gnc}$; see~\cite[Theorem 3.1]{ObVo15}. 
The second formula reflects the fact that the Lazard correspondence induces an order-preserving correspondence between subgroups of $\G_N$ and sublattices of $\gn$, and maps
normal subgroups to ideals. Moreover, centralizers of elements in~$\G_N$ correspond to centralizers of elements in~$\gn$ under the Lazard correspondence.

The cardinalities of $\gn$ and $\gn/\zn$ are powers of $q$, and hence so are the cardinalities of $\rad/\z$ and $\Ker(\ad_{x}^{\star})$.
It follows that $r_{n}(\G_N)$ and $c_n(\G_N)$ can only be nonzero if $n$ is a power of $q$.

The next step is to relate~\eqref{chi} and~\eqref{cci} to the commutator matrices of~$\gn$ with respect to~$\mathbf{e}_N$ and $\mathbf{f}_N$.

Using arguments analogous to the ones of~\cite[Section~2]{ObVo15}, we define the coordinate systems
\begin{align*} \phi_N: \gn/\zn &\to (\lrip)^a, && x=\sum_{j=1}^{a}x_je_{j}^{N} \mapsto \mathbf{x}=(x_1, \dots, x_a),&\\
 \psi_N:\widehat{\gnc} &\to (\lrip)^b, && \omega=\sum_{j=1}^{b}y_jf_{j}^{N\vee} \mapsto \mathbf{y}=(y_1, \dots, y_b),&
\end{align*}
where, for $N \in \N_0$, $\mathbf{f}_{N}^{\vee}=(f_{1}^{N\vee}, \dots, f_{b}^{N\vee})$ is the dual $\lri/\p$-basis for \[\widehat{\gnc}=\Hom_{\lri}(\gnc,\C^{\times}).\]
We notice that $\g_1/\z_1$ and $\g'_1$ are regarded as $\lri/\p$-vector spaces in the construction of~\cite[Section~2]{ObVo15}. 
In the coordinate systems above, we regard $\gn/\zn$ and $\gnc$ as $\lri/\p^N$-modules for all $N\in \N$.

\begin{lem}\label{ObVo3.3} Given $x \in \gn/\zn$ with $\phi_N(x)=\mathbf{x}$, and $\omega \in \widehat{\gnc}$ with $\psi_N(\omega)=\mathbf{y}$, the following holds.
\begin{align*}
 x \in \rad/\zn \text{ if and only if } B(\mathbf{y})\mathbf{x}^{\textup{tr}}=0,\\
 \omega \in \Ker(\ad_{x}^{\star}) \text{ if and only if } A(\mathbf{x})\mathbf{y}^{\textup{tr}}=0.
\end{align*}
\end{lem}
\begin{proof} 
An element $x \in \gn/\zn$ belongs to $\rad/\z$ exactly when $\omega[v,x]=1$, for all $v \in \gn/\zn$, whilst an element $\omega\in \widehat{\gnc}$ belongs to 
$\Ker(\ad_{x}^{\star})$ exactly when $\omega[v,x]=1$ for all $v \in \gn/\zn$. Expressing these conditions in coordinates, we see that both expressions hold.

We prove the second claim in detail.
Fix $x \in \gn/\zn$ with \[\phi_N(\overline{x})=\mathbf{x}=(x_1, \dots, x_a).\] 
It holds that 
\begin{equation}\label{eix}
\left[e_{i}^{N},x\right]=\sum_{j=1}^{a}x_j[e_{i}^{N},e_{j}^{N}]=\sum_{j=1}^{a}\sum_{l=1}^{b}\lambda_{ij}^{l}x_jf_{l}^{N}.  
\end{equation}
We want to determine which elements $\omega\in \widehat{\gnc}$ satisfy $\omega([v,x])=1$, for all $v\in\gn/\zn$.  
Consider $\psi_N(\omega)=\mathbf{y}=(y_1, \dots , y_b)$, i.e., $\omega=\prod_{k=1}^{b}(f_{k}^{N\vee})^{y_k}$. Because of~\eqref{eix}, for each $i\in[a]$, 
\[\omega([e_{i}^{N},x])=\prod_{k=1}^{b}\left(f_{k}^{N\vee}\left(\sum_{j=1}^{a}\sum_{l=1}^{b}\lambda_{ij}^{l}x_jf_{l}^{N}\right)\right)^{y_k}=
\prod_{k=1}^{b}\left(f_{k}^{N\vee}\left(f_{k}^{N}\right)\right)^{y_k\sum_{j=1}^{a}\lambda_{ij}^{k}x_j}.\]
This expression equals $1$ exactly when $\sum_{k=1}^{b}y_k\sum_{j=1}^{a}\lambda_{ij}^{k}x_j=0$. Now, by definition, $\sum_{j=1}^{a}\lambda_{ij}^{k}x_j=A(\mathbf{x})_{ik}$, where 
$A(\mathbf{x})$ is the $A$-commutator matrix of Definition~\ref{commutator} evaluated at $\mathbf{x}$.
Consequently, $\omega \in \Ker(\ad_{x}^{\star})$ if and only if 
\[\sum_{k=1}^{b}A(\mathbf{x})_{ik}y_k=0,\text{ for all }i\in [a],\] 
that is, $A(\mathbf{x})\mathbf{y}^{\textup{tr}}=0$.
\end{proof}

Applying Lemma~\ref{ObVo3.3} to~\eqref{chi}, we rewrite the numbers $r_{q^i}(\G_N)$ in terms of solutions of the system 
$B(\mathbf{y})\mathbf{x}^{\textup{tr}}=0$ and, applying Lemma~\ref{ObVo3.3} to~\eqref{cci}, we rewrite the numbers $\cc_{q^i}(\G_N)$ in terms of solutions of the system $A(\mathbf{x})\mathbf{y}^{\textup{tr}}=0$. 
In each case, we consider the elementary divisor type of the corresponding matrix.

Fix an elementary divisor type $\tilde{\nu}(B(\mathbf{y}))=(m_1, \dots, m_{\ub})$, where 
\[2\ub=\max\{\rk_{\text{Frac}(\lri)}B(\mathbf{z}) \mid \mathbf{z}\in \lri^b\}.\]
Since $B(\mathbf{y})$ is similar to the matrix $\textup{Diag}(\pi^{m_1},\pi^{m_1}, \dots,\pi^{m_{\ub}},\pi^{m_{\ub}},\mathbf{0}_{a-2\ub})$, where $\mathbf{0}_{a-2\ub}=(0,\dots,0)\in \Z^{a-2\ub}$, 
the system $B(\mathbf{y})\mathbf{x}^{\textup{tr}}=0$ in $\lri/\p^N$ is 
equivalent to
\[\begin{cases}	x_1\phantom{2u_b-} \equiv x_2\phantom{ub} \equiv &0 ~\bmod \p^{N-m_1},\\ 
		x_3\phantom{2u_b-} \equiv x_4\phantom{ub} \equiv &0 ~\bmod \p^{N-m_2},\\
		&\vdots \\ 
		x_{2\ub-1} \equiv x_{2\ub} \equiv &0 ~\bmod \p^{N-m_{\ub}}.
\end{cases}\]

For $2\ub <a$, the elements $x_{2\ub+1}, \dots , x_a$ are arbitrary elements of $\lri/\p^N$, and
\[|\{x \in \lri/\p^N \mid x \equiv 0 \bmod \p^{N-m_j}\}|=q^{m_j}.\] 
Hence, the number of solutions of $B(\mathbf{y})\mathbf{x}^{\textup{tr}}=0$ in $\lri/\p^N$ is 
\[q^{2(m_1+ \dots +m_{\ub})+(a-2\ub)N}.\]
In other words, $\tilde{\nu}(B(\mathbf{y}))=(m_1, \dots, m_{\ub})$ implies \[|\rad / \zn|=q^{2(m_1+ \dots + m_{\ub})+(a-2\ub)N}.\]
Lemma~\ref{ObVo3.3} then assures that $|\rad / \zn|=q^{-2i}|\gn/\zn|=q^{aN-2i}$, whenever $B(\mathbf{y})$ has elementary divisor type $(m_1, \dots, m_{\ub})$ satisfying 
\[\sum_{j=1}^{\ub}m_j=\ub N-i.\]

Consequently, for $r=\rk(\g/\g')=h-b$, expression~\eqref{chi} can be rewritten as 
\begin{equation}\label{chin} r_{q^i}(\G_N)=\sum_{\m \in \Dr_{B}^{N}}
|\{\mathbf{y} \in (\lri/\p^N)^b \mid \tilde{\nu}(B(\mathbf{y}))=\m\}|q^{rN-2i},
\end{equation}
where 
 \[\Dr_{B}^{N}:=\left\{\m=(m_1, \dots, m_{\ub}) \in \N_{0}^{\ub} \bigm| m_1 \leq \dots \leq m_{\ub} \leq N,~\textstyle\sum\limits_{i=1}^{\ub}m_i=\ub N-i\right\}.\]
Analogously, if $\nu(A(\mathbf{x}))=(m_1, \dots, m_{\ua})$, where 
\[\ua:=\max\{\rk_{\text{Frac}(\lri)}A(\mathbf{z}) \mid \mathbf{z}\in \lri^a\},\] 
the equality $A(\mathbf{x})\mathbf{y}^{\textup{tr}}=0$ has $q^{m_1+m_2 +\dots +m_{\ua}+(b-\ua)N}$ solutions in $\lri/\p^N$. For $z=\rk(\z)=h-a$, this yields
\begin{equation}\label{ccin} c_{q^i}(\G_N)=
\sum_{\m \in \Dr_{A}^{N}}
|\{\mathbf{x} \in (\lri/\p^N)^a \mid {\nu}(A(\mathbf{x}))=\m\}|q^{zN-i},
\end{equation}
where 
 \[\Dr_{A}^{N}:=\left\{\m=(m_1, \dots, m_{\ua}) \in \N_{0}^{\ua} \bigm| m_1 \leq \dots \leq m_{\ua} \leq N,~\textstyle\sum\limits_{i=1}^{\ua}m_i=\ua N-i\right\}.\]

For a matrix $\mathcal{R}(\underline{Y})=\mathcal{R}(Y_1, \dots, Y_n)$ of polynomials as the one at the beginning of Section~\ref{padicintegral} and for 
$\mathbf{m}=(m_1, \dots, m_{\ur})\in \N_{0}^{\ur}$, define
\[ \mathfrak{W}_{\m}^{\mathcal{R}}(\lrip):=\{\mathbf{y} \in (\lri/\p^N)^n \mid \vry=\mathbf{m}\}.\]
Expressions~\eqref{chin} and~\eqref{ccin} are written in terms of cardinalities of such sets, which are related to the numbers $\Nr_{\m}(\lrip)$ as follows.
Write 
\[\mathbf{m}-m=(m_1-m, \dots, m_{\ur}-m)\text{ for all }m \in \N_{0}.\]
If $\mathcal{R}(\mathbf{y})$ is such that $v_\p(\mathbf{y})=v_{\p}(\mathcal{R}(\mathbf{y}))$ for all $\mathbf{y} \in \lri^n$, then
\begin{equation}\label{wm} |\mathfrak{W}_{\m}^{\mathcal{R}}(\lrip)|= \Nr_{\m-m_1}(\lri/\p^{N-m_1}).\end{equation}
Indeed, the map $\mathfrak{N}_{\m-m_1}^{\mathcal{R}}(\lri/\p^{N-m_1}) \to \mathfrak{W}_{\m}^{\mathcal{R}}(\lrip)$ given by $\mathbf{y}\mapsto \pi^{m_1}\mathbf{y}$ is a bijection.
Equality~\eqref{wm} provides the following reformulations of~\eqref{chin} and~\eqref{ccin}.
\hspace{-0.5cm}
\begin{lem}\label{chin2} 
For each $i \in \N_0$ and $N\in \N_0$,
\begin{align}
\label{lemch} r_{q^i}(\G_N)&=\sum_{\m \in \mathcal{D}_{B}^{N}}\Nb_{\m-m_1}(\lri/\p^{N-m_1})q^{rN-2i},\\
\label{lemcc} c_{q^i}(\G_N)&=\sum_{\m \in \mathcal{D}_{A}^{N}}\Na_{\m-m_1}(\lri/\p^{N-m_1})q^{zN-i}.
\end{align}
\end{lem}
\begin{rmk}\label{condition2}
 As explained in Section~\ref{groups}, for $\ri$-Lie lattices $\Lambda$ of nilpotency class~$2$, a different construction for~$\G=\G_{\Lambda}$ is given in~\cite[Section~2.4.1]{StVo14} which does not require the assumption $\Lambda' \subseteq 2\Lambda$. 
 For groups associated to such schemes a Kirillov orbit method formalism was formulated which is valid for all primes; see~\cite[Section~2.4.2]{StVo14}. 
 Consequently, \cite[Lemma~2.13]{StVo14} assures that~\eqref{lemch} holds for all primes~$\p$ if~$\Lambda$ has nilpotency class~$2$. 
\end{rmk}

\subsection{\texorpdfstring{$\p$}{p}-adic integrals}\label{spadic}
We now write the local factors of the bivariate zeta functions of $\G(\ri)$ in terms of Poincaré series such as~\eqref{poincare}. 

Recall from Section~\ref{chcc} that the dimensions of irreducible complex representations as well as the sizes of the conjugacy classes of 
$\G(\lri)$ are powers of $q$, allowing us to write the representation and the conjugacy class zeta functions of the congruence quotient $\G_N=\G(\lrip)$ as  
\[\rzeta_{\G_N}(s)=\sum_{i=0}^{\infty}r_{q^i}(\G_N)q^{-is}~\text{ and }~\cczeta_{\G_N}(s)=\sum_{i=0}^{\infty}c_{q^i}(\G_N)q^{-is}.\]
These sums are finite, since $\G_N$ is a finite group.
Applying this to~\eqref{localfactors}, the definition of the local factors of the bivariate zeta functions, one obtains 
\begin{align*}
 \rlvlzf_{\G(\lri)}(s_1,s_2)&=\sum_{N=0}^{\infty}\sum_{i=0}^{\infty}r_{q^i}(\G_N)q^{-is_1-Ns_2} \text{ and }\\
 \clvlzf_{\G(\lri)}(s_1,s_2)&=\sum_{N=0}^{\infty}\sum_{i=0}^{\infty}\cc_{q^i}(\G_N)q^{-is_1-Ns_2}.
\end{align*}
Recall that $z=\rk(\z)=h-a$ and $r=\rk(\g/\g')=h-b$. If $c=2$ or $p>c>2$, then~\eqref{lemch} yields 
\begin{equation}\label{rlocal}\rlvlzf_{\G(\lri)}(s_1,s_2)=\sum_{N=0}^{\infty}\sum_{i=0}^{\infty}\sum_{\m\in \mathcal{D}_{B}^{N}}
\Nb_{\m-m_1}(\lri/\p^{N-m_1})q^{-(s_2-r)N-(2+s_1)i}.
\end{equation}
If $p>c$, then~\eqref{lemcc} yields
\begin{equation}\label{cclocal}\clvlzf_{\G(\lri)}(s_1,s_2)= 
\sum_{N=0}^{\infty}\sum_{i=0}^{\infty}\sum_{\m \in \mathcal{D}_{A}^{N}}
\Na_{\m-m_1}(\lri/\p^{N-m_1})q^{-(s_2-z)N-(1+s_1)i}. 
\end{equation}
We now show how to rewrite these sums as Poincaré series of the form~\eqref{poincare}. In preparation for this, we need two lemmata. 
\begin{lem}\label{terms} Let $s$ be a complex variable, $(a_m)_{m\in \N_0}$ a sequence of real numbers, and $q\in \Z_{\geq 2}$. Provided both series converge, the following holds:
 \[\sum_{N=1}^{\infty}\sum_{m=0}^{N-1}a_mq^{-sN}= \frac{q^{-s}}{1-q^{-s}}\left(\sum_{N=0}^{\infty}a_Nq^{-sN}\right).\]
\end{lem}
\begin{proof} This is due to a short manipulation of geometric series; see~\cite[Lemma~3.2.9]{PLphd}.
\end{proof}
\begin{lem}\label{rewr} Let $s$ and $t$ be complex variables. Let $\mathcal{R}(\underline{Y})=\mathcal{R}(Y_1, \dots, Y_n)$ be a matrix of polynomials 
$\mathcal{R}(\underline{Y})_{ij}\in\lri[\underline{Y}]$. If $\mathcal{R}$ is not antisymmetric, set $u=\max\{\rk_{\Frac(\lri)}\mathcal{R}(\mathbf{z}) \mid \mathbf{z}\in \lri^n\}$. Otherwise, set $u=\tfrac{1}{2}\max\{\rk_{\Frac(\lri)}\mathcal{R}(\mathbf{z}) \mid \mathbf{z}\in \lri^n\}$. Moreover, let $q=|\lri/\p|$. Provided both series converge, the following holds:
 \begin{align}\label{sumlema}&\sum_{N=0}^{\infty}\sum_{i=0}^{\infty}\sum_{\m \in \Dr^N(uN-i)}
\Nr_{\m-m_1}(\lri/\p^{N-m_1})q^{-sN-ti}\\
 &=\frac{1}{1-q^{-s}}\left(1+\sum_{N=1}^{\infty}\sum_{\m \in \N_{0}^{u}}\hspace{-0.2cm}\Nr_{\m}(\lrip)q^{-(s+ut)N+t \sum\limits_{j=1}^{u}m_j}\right),\nonumber
 \end{align}
where for each $c \in \N_0$, 
\[\Dr^N(c):=\left\{\m=(m_1,\dots,m_u)\in \N_{0}^{u} \mid m_1\leq \dots \leq m_u \leq N,\textstyle\sum\limits_{i=1}^{u}m_i=c\right\}.\] 
 \end{lem} 
\begin{proof}
Set $\mathbf{m}=(m_1, \dots, m_u)$ and recall the notation 
\[\mathbf{m}-m=(m_1-m,\dots, m_u-m)\text{ for }m \in \N_0.\]

As $\Nr_{\m}(\lrip)=0$ unless $0=m_1 \leq m_2 \leq \dots \leq m_{u} \leq N$, in which case \
\[0\leq \textstyle\sum_{j=1}^{u}m_j \leq uN,\] 
the condition $\textstyle\sum_{j=1}^{u}m_j=uN-i$ implies that the only values of~$i$ which are relevant for the sum~\eqref{sumlema} are $0 \leq i \leq uN$.
Hence, the expression on the left-hand side of~\eqref{sumlema} can be rewritten as
\begin{equation}\label{rewr2} 1+\sum_{N=1}^{\infty}\sum_{i=0}^{uN}\sum_{\m\in\Dr^N(uN-i)} 
\Nr_{\m-m_1}(\lri/\p^{N-m_1})q^{-sN-ti}.
\end{equation}
In the following, we make use of the notation \[\Dl^N(c):=\left\{\m=(m_1,\dots,m_u)\in \N_{0}^{u} \mid m_1\leq \dots \leq m_u \leq N,\textstyle\sum\limits_{i=1}^{u}m_i\leq c\right\}.\] 
Restricting the summation in~\eqref{rewr2} to $m_1=0$ yields
\[\sum_{N=1}^{\infty}\sum_{\substack{\m \in \Dl^N((u-1)N)\\ m_1=0}}
\Nr_{\m}(\lrip)q^{-sN-t(uN-\sum\limits_{j=2}^{u}m_j)}.\]
Since $\Nr_{\m}(\lrip)=0$ unless $0=m_1 \leq m_2 \leq \dots \leq m_{u} \leq N$, we may rewrite this sum as 
\begin{equation*} \sum_{N=1}^{\infty}\sum_{\m \in \N_{0}^{u} }\Nr_{\m}(\lrip)q^{-(s+ut)N+t\sum\limits_{j=1}^{u}m_j}=:\mathcal{S}(s,t).
\end{equation*}
Our goal now is to write the part of the summation in~\eqref{rewr2} with $m_1> 0$ in terms of $\mathcal{S}(s,t)$. 
Restricting the summation in~\eqref{rewr2} to $m_1>0$ yields
\begin{align}
& \sum_{N=1}^{\infty}\sum_{i=0}^{u N}\sum_{\substack{\m \in \Dr^N(uN-i)\\ m_1> 0}}
\Nr_{\m-m_1}(\lri/\p^{N-m_1})q^{-sN-ti}\nonumber \\
\label{change}&= \sum_{N=1}^{\infty}\sum_{m=1}^{N}\sum_{\substack{\m \in \Dl^N(uN)\\ m_1=m}}
\Nr_{\m-m_1}(\lri/\p^{N-m_1})q^{-sN-t(uN-\sum\limits_{j=1}^{u}m_j)}.
\end{align}
By writing $m'_j=m_j-m_1$, we obtain $\textstyle\sum_{j=1}^{u}m_j=um_1+\sum_{j=2}^{u}m'_j$. Moreover, for each $m \in [N]$,
\[\{\m-m \mid \m \in \Dl^N(uN),~ m_1=m\}=\{\m \in \Dl^{N-m}(u(N-m)) \mid m_1=0\}.\] Then, we may rewrite~\eqref{change} as 
\begin{align}
&\sum_{N=1}^{\infty}\sum_{m=1}^{N}\sum_{\substack{\m \in \Dl^{N-m}(u(N-m))\\ m_1=0}}
 \Nr_{\m}(\lri/\p^{N-m})q^{-sN-t(u(N-m)-\sum\limits_{j=2}^{u}m_{j})}\nonumber\\
 &=\sum_{N=1}^{\infty}q^{-sN}\sum_{m=0}^{N-1}\sum_{\substack{\m \in \Dl^{m}(um)\\ m_1=0}}
 \Nr_{\m}(\lri/\p^m)q^{t\sum\limits_{j=2}^{u}m_{j}-tum} \nonumber\\ 
&=\label{notzero}\sum_{N=1}^{\infty}q^{-sN}\sum_{m=0}^{N-1}\sum_{\m \in \N_{0}^{u}}
\Nr_{\m}(\lri/\p^m)q^{t\sum\limits_{j=1}^{u}m_j-tum}.
\end{align}
Apply Lemma~\ref{terms} to~\eqref{notzero} by setting
\[a_m:=\sum_{\m \in \N_{0}^{u}}\Nr_{\m}(\lri/\p^m)q^{t\sum_{j=1}^{u}m_j-tum}.\]
This gives that~\eqref{notzero} equals 
\begin{align*}
&\frac{q^{-s}}{1-q^{-s}}\left(1+\sum_{N=1}^{\infty}\sum_{\m \in \N_{0}^{u}}\Nr_{\m}(\lrip)q^{-(s+ut)N+t\sum\limits_{j=1}^{u}m_j}\right) 
&=\frac{q^{-s}}{1-q^{-s}}\left(1+\mathcal{S}(s,t)\right).
 \end{align*}
Combining the expressions for the parts of the sum with $m_1=0$ and $m_1> 0$ yields
\begin{align*}&\sum_{N=0}^{\infty}\sum_{i=0}^{\infty}\sum_{\m \in \Dr^N(uN-i)}
\Nr_{\m}(\lrip)q^{-sN-ti}\\
&=1+\mathcal{S}(s,t)+\frac{q^{-s}}{1-q^{-s}}\left( 1+\mathcal{S}(s,t)\right)=\frac{1}{1-q^{-s}}\left(1+\mathcal{S}(s,t)\right).\qedhere
 \end{align*}
\end{proof}

\begin{pps} If either $c=2$ or $p>c>2$, then 
\begin{align}\label{rint} &\rlvlzf_{\G(\lri)}(s_1,s_2)=\\ \frac{1}{1-q^{r-s_2}}
&\left(1+\sum_{N=1}^{\infty}\sum_{\m\in\N_{0}^{\ub}} \Nb_{\m}(\lrip)q^{-N(\ub s_1+s_2+2\ub-r)-2\sum\limits_{j=1}^{\ub}m_j\frac{(-s_1-2)}{2}}\right).\nonumber 
\end{align}
Moreover, if $p>c$, then 
\begin{align}
\label{cint} &\clvlzf_{\G(\lri)}(s_1,s_2)= \\\frac{1}{1-q^{z-s_2}}
&\left(1+\sum_{N=1}^{\infty}\sum_{\m\in\N_{0}^{\ua}}\Na_{\m}(\lrip)q^{-N(\ua s_1+s_2+\ua-z)-\sum\limits_{j=1}^{\ua}m_j(-s_1-1)}\right).\nonumber
\end{align}
\end{pps}
\begin{proof}
By setting $s=s_2-r$ and $t=2+s_1$, and considering $\mathcal{R}$ to be the $B$-commutator matrix of $\Lambda$ on the left-hand side of~\eqref{sumlema}, we obtain~\eqref{rlocal}. 
Under these substitutions, Lemma~\ref{rewr} shows~\eqref{rint}.
Analogously, by setting $s=s_2-z$ and $t=1+s_1$, and considering $\mathcal{R}$ to be the $A$-commutator matrix of $\Lambda$, the left-hand side of~\eqref{sumlema} equals~\eqref{cclocal}, so that, under these substitutions, Lemma~\ref{rewr} shows~\eqref{cint}.
\end{proof}
Expression~\eqref{rint} is of the form~\eqref{poincareas} with 
\[t=\ub s_1+s_2+2\ub-r\text{ and }r_k=\frac{-s_1-2}{2}\text{ for each }k \in [\ub],\] 
whilst~\eqref{cint} is~\eqref{poincare} with 
\[t=\ua s_1+s_2+\ua-z\text{ and }r_k=-s_1-1\text{ for each }k \in [\ua].\] 
Therefore these choices of $t$ and $\underline{r}$ applied to~\eqref{Vo10as} and to~\eqref{Vo14gen} yield the following.
Recall $a+z=\rk(\g/\z)+\rk(\z)=\rk(\g)=h$ and $b+r=\rk(\g')+\rk(\g/\g')=h$. For $k \in \N$ write $\mathbf{1}_{k}=(1, \dots , 1) \in \Z^k$. 
\begin{pps}\label{intpadic} If either $c=2$ or $p>c>2$, then 
\begin{align}
\label{rpadic} &\rlvlzf_{\G(\lri)}(s_1,s_2)=
&\frac{1}{1-q^{r-s_2}}\left(1+\pint_{B}\left(\left(\tfrac{-s_1-2}{2}\right)\mathbf{1}_{\ub},\ub s_1+s_2+2\ub-h-1\right)\right). 
\end{align}
Moreover, if $p>c$, then 
\begin{align}
\label{ccpadic} &\clvlzf_{\G(\lri)}(s_1,s_2)=
\frac{1}{1-q^{z-s_2}}\left(1+\pint_{A}\left((-s_1-1)\mathbf{1}_{\ua},\ua s_1+s_2+\ua-h-1\right)\right).
\end{align}
\end{pps}
Specialization~\eqref{specialization} applied to~\eqref{rpadic} and to~\eqref{ccpadic} yields 
\begin{align}\label{cctok}\kzeta_{\G(\lri)}(s)&=
\frac{1}{1-q^{z-s}}\left(1+\pint_{A}\left(-\mathbf{1}_{\ua},s+\ua-h-1\right)\right),\\
&\label{rtok}=\frac{1}{1-q^{r-s}}\left(1+\pint_{B}\left(-\mathbf{1}_{\ub},s+2\ub-h-1\right)\right).
\end{align}

\begin{rmk}
Formula~\eqref{cctok} coincides with the $\p$-adic integral obtained from the $\p$-adic integral~\cite[formula~(4.3)]{Ro17} together with the 
specialization given in \cite[Theorem~1.7]{Ro17}.

In fact, for each $x \in \g$, let $\ad_x:\g \to \g'$ be the adjoint homomorphism 
\[\ad_{x}(z)=[z,x],\text{ for all }z \in \g.\] 
As in Section~\ref{chcc}, let $\mathscr{B}=(e_1, \dots, e_h)$ be a basis of $\g$ with the properties described there; we use the notation that was set up in this context.
For each $x \in \g$, we can write $x=\sum_{i=1}^{h}x_ie_i$, for some $x_i \in \lri$. Let $\mathbf{x}=(x_1, \dots, x_h) \in \lri^h$.
The $b \times h$-matrix representing the linear transformation $\ad_x$ is such that its submatrix composed of its first $a$ columns is the transpose $A(\mathbf{x})^{\textup{tr}}$ of the $A$-commutator matrix of~$\Lambda$, and the remaining columns have only zero entries. 

We observe that the above mentioned integrals of~\cite{Ro17} are taken over $\lri \times \lri^a$ instead of $\p \times \W_a$ as in~\eqref{cctok}. Formula~\eqref{cctok} coincides with the above mentioned $\p$-adic integral due to~\cite[Lemma~2.2.4]{PLphd}.
\end{rmk} 

\begin{exm}\label{Heis}
Let $\mathbf{H}(\ri)$ be the Heisenberg group over~$\ri$ considered in Example~\ref{exintro}. 
The unipotent group scheme~$\mathbf{H}$ is obtained from the $\Z$-Lie lattice \[\Lambda=\langle x_1,x_2, y\mid [x_1,x_2]-y\rangle.\] 
The commutator matrices of $\g=\Lambda(\lri)$ with respect to the ordered sets $\mathbf{e}=(x_1,x_2)$ and $\mathbf{f}=(y)$ are
\[ A(X_1, X_2)=\left[ \begin {array}{c} X_2\\ \noalign{\medskip} -X_1\end {array}  \right] ~\text{ and }~ B(Y)=\left[ \begin {array}{cc} 0&Y\\ \noalign{\medskip}-Y&0\end {array}
\right].\]

The $A$-commutator matrix has rank~$1$ and the $B$-commutator matrix has rank~$2$ over the respective fields of rational functions, that is, $\ua=\ub=1$. 
Moreover, $h=\rk(\g)=3$, and 
\[F_{1}(A(X_1,X_2))=\{-X_1,X_2\}, \hspace{0.5cm} F_2(B(Y))=\{Y^2\}.\] 
In particular, if $(x_1,x_2)\in \W_{2}$, i.e., $v_{\p}(x_1,x_2)=0$, then $\|F_1(A(x_1,x_2))\|_{\p}=1$.
Also, if $y \in \W_1$, then, in particular, $v_{\p}(y^2)=0$, which gives $\|F_2(B(y))\|_{\p}=1$. 

It follows from Proposition~\ref{intpadic} that 
\begin{align*}
  \rlvlzf_{\mathbf{H}(\lri)}(s_1,s_2)	&=\frac{1}{1-q^{2-s_2}}\left(1+(1-q^{-1})^{-1}\int_{(w, y) \in \p \times \W_{1}}|w|_{\p}^{s_1+s_2-2}d\mu\right),\\
  \clvlzf_{\mathbf{H}(\lri)}(s_1,s_2)	&=\frac{1}{1-q^{1-s_2}}\left(1+(1-q^{-1})^{-1}\int_{(w, x_1,x_2) \in \p \times \W_{2}}|w|_{\p}^{s_1+s_2-3}d\mu\right).\\
\end{align*}
Expressions~\eqref{Heisirr} and~\eqref{Heiscc} for $\rlvlzf_{\mathbf{H}(\lri)}(s_1,s_2)$ and $\clvlzf_{\mathbf{H}(\lri)}(s_1,s_2)$ given in Example~\ref{exintro} are then consequence of the following well-known fact: for $k \in \N$ and $t \in \C$, 
\begin{equation}\label{integral1} 
 \int_{w\in\p^k}|w|_{\p}^{t}d\mu=\frac{q^{-k(t+1)}(1-q^{-1})}{1-q^{-k(t+1)}},
\end{equation}
provided the $\p$-adic integral on the left-hand side converges.
\end{exm}

\subsection{Twist representation zeta functions}\label{irr}
In this section, we assume that~$\G$ is the unipotent group scheme associated to a nilpotent $\ri$-Lie lattice~$\Lambda$ of nilpotency class~$2$ without the assumption $\Lambda'\subseteq 2\Lambda$, constructed as in~~\cite[Section~2.4.1]{StVo14}.
We provide a univariate specialization of the bivariate representation zeta function of $\G(\lri)$ which results in the twist representation zeta function of this group. 

According to~\cite[Corollary~2.11]{StVo14}, the twist representation zeta function of $\G(\lri)$ is given by
\[\rtzeta_{\G(\lri)}(s)=1+\pint_{B}(-\tfrac{s}{2}\mathbf{1}_{\ub}, u_Bs-b-1),\]
where $b=\rk(\g')$, $2u_B=\max\{\rk_{\text{Frac}(\lri)}B(\mathbf{z}) \mid \mathbf{z}\in \lri^b\}$, $\mathbf{1}_{\ub}=(1, \dots, 1) \in \Z^{\ub}$, and $\pint_{B}(\underline{r},t)$ is the integral 
$\pint_{\mathcal{R}}(\underline{r},t)$ given in~\eqref{Vo10as} with $\mathcal{R}(\underline{Y})$ being regarded as the $B$-commutator matrix $B(\underline{Y})$ of~$\g$. 
Recall that $r=\rk(\g/\g')=h-b$. Proposition~\ref{intpadic} states  
\[(1-q^{r-s_2})\rlvlzf_{\G(\lri)}(s_1,s_2)=1+\pint_{B}\left(\tfrac{-2-s_1}{2}\mathbf{1}_{\ub},\ub s_1+s_2+2\ub-h-1\right).\]
Comparing the expressions for $\rtzeta_{\G(\lri)}(s)$ and $(1-q^{r-s_2})\rlvlzf_{\G(\lri)}(s_1,s_2)$, we obtain the desired specialization. 
\begin{pps}\label{repzeta} If $\G(\lri)$ has nilpotency class~$2$, then 
\begin{equation*}
(1-q^{r-s_2})\rlvlzf_{\G(\lri)}(s_1,s_2)\mid_{\substack{s_1\to s-2\\ s_2\to r\phantom{-2}}}=\rtzeta_{\G(\lri)}(s),
\end{equation*}
provided both the left-hand side and the right-hand side converge.
\end{pps}
In the following example, we exhibit a $\mathcal{T}$-group of nilpotency class~$3$ whose bivariate representation zeta function does not specialize to its twist representation zeta function. 
\begin{exm}\label{irr3}
Consider the following free nilpotent $\Z$-Lie lattice on~$2$ generators of class~$3$ 
\[\f_{3,2}=\langle x_1,x_2,y,z_1,z_2 \mid i,j\in \{1,2\}: ~[x_1,x_2]-y,~[y,x_i]-z_i, ~[z_i,x_j], ~[z_i,y], ~[z_1,z_2]  \rangle.\]
Let~$\mathfrak{F}_{3,2}$ denote the unipotent group scheme obtained from~$\f_{3,2}$, and denote by~$\z_{3,2}$ and by~$\f'_{3,2}$ the center and 
the derived Lie lattice of~$\f_{3,2}$, respectively. 

The $B$-commutator matrix of~$\f_{3,2}$ with respect to $\mathbf{e}=(y,x_1,x_2)$ and $\mathbf{f}=(z_1,z_2,y)$ is
\[
B(Y_1,Y_2,Y_3)=\left[ \begin {array}{ccc} 0&Y_1&Y_2\\ \noalign{\medskip}-Y_1&0&Y_3
\\ \noalign{\medskip}-Y_2&-Y_3&0\end {array} \right] .
\]
Thus, $\ub=1$, $F_0(B(\underline{Y}))=\{1\}$, and $F_2(B(\underline{Y}))\supseteq \{Y_{1}^{2},Y_{2}^{2},Y_{3}^{2}\}$. 
It follows from Proposition~\ref{intpadic} and~\eqref{integral1} that
\begin{align}	
 \rlvlzf_{\mathfrak{F}_{2,3}(\lri)}(s_1,s_2)&=\frac{1}{1-q^{2-s_2}}\left(1+(1-q^{-1})^{-1}\int_{(w,y_1,y_2,y_3)\in\p\times \W_{3}}|w|_{\p}^{s_1+s_2-4}d\mu\right)\nonumber \\
 \label{lvlf23}&=\frac{1-q^{-s_1-s_2}}{(1-q^{2-s_2})(1-q^{3-s_1-s_2})}. \hspace{0.5cm}
\end{align}
By implementing his methods in \textup{Zeta}~\cite{RoZeta}, Rossmann provides in~\cite[Table~1]{Ro17comp} the following formula for the twist representation zeta function of $\mathfrak{f}_{3,2}$---denoted by $L_{5,9}$ in~\cite{Ro17comp}---,~provided $q$ is sufficiently large 
\begin{equation}\label{rf23}\rtzeta_{\mathfrak{F}_{3,2}(\lri)}(s)=\frac{(1-q^{-s})^2}{(1-q^{1-s})(1-q^{2-s})}. 
\end{equation}

Comparing~\eqref{lvlf23} and~\eqref{rf23}, we see that there is no specialization of form~\eqref{specializationirr} for the bivariate representation zeta function of~$\mathfrak{F}_{3,2}(\lri)$ that leads to its twist representation zeta function.

For the sake of completeness, we now calculate the bivariate conjugacy class and the class number zeta functions of~$\mathfrak{F}_{3,2}(\lri)$.
The $A$-commutator matrix of~$\f_{3,2}$ with respect to~$\mathbf{e}$ and~$\mathbf{f}$ is
\[A(X_1,X_2,X_3)= \left[ \begin {array}{ccc} X_2&X_3&0\\ \noalign{\medskip}-X_1&0&X_3
\\ \noalign{\medskip}0&-X_1&-X_2\end {array} \right].\]
Thus, $\ua=2$, $F_0(A(\underline{X}))=\{1\}$, $F_1(A(\underline{X}))=\{-X_1,\pm X_2,X_3\}$, and $F_2(A(\underline{X}))\supseteq \{X_{1}^{2},-X_{2}^{2},X_{3}^{2}\}$.
Hence
\begin{align*}	
 \clvlzf_{\mathfrak{F}_{2,3}(\lri)}(s_1,s_2)&=\frac{1}{1-q^{2-s_2}}\left(1+(1-q^{-1})^{-1}\int_{(w,x_1,x_2,x_3)\in\p\times \W_{3}}|w|_{\p}^{2s_1+s_2-4}d\mu\right)\\
 &=\frac{1-q^{-2s_1-s_2}}{(1-q^{2-s_2})(1-q^{3-2s_1-s_2})}.
\end{align*}

Specialization~\eqref{specialization} yields 
\[\kzeta_{\mathfrak{F}_{2,3}(\lri)}(s)=\frac{1-q^{-s}}{(1-q^{2-s})(1-q^{3-s})}.\] 
This formula agrees with the one given in~\cite[Section~9.3, Table~1]{Ro17}. 
\end{exm}

\subsection{Local functional equations---proof of Theorem~\ref{thmA}}\label{pthmA}
Proposition~\ref{intpadic} assures that, for each $\ast \in \{\irrup, \ccup\}$, almost all local factors of~$\astlvlzf_{\G(\ri)}$ are given by a rational function~$R^{\ast}$ in certain parameters, as stated in Theorem~\ref{thmA}. 

In this section, we conclude the proof of Theorem~\ref{thmA} by showing that the integrals given in Proposition~\ref{intpadic} behave uniformly under base extension and that they satisfy functional equations. 

Fix a nonzero prime ideal~$\p$ satisfying the conditions of Proposition~\ref{intpadic}. 

Let~$L$ be a finite extension of $K=\Frac(\ri)$ with ring of integers~$\ri_L$. For a fixed prime ideal~$\Pp$ of~$\ri_L$ dividing~$\p$, write~$\Lri$ for the localisation~$\ri_{L,\Pp}$. 
Denote the relative degree of inertia by $f=f(\Lri,\lri)$, and hence $|\Lri / \Pp|=q^{f}$. 
Set~$\g_{L}=\Lambda(\Lri)$, and let~$\z_{L}$ and~$\g'_{L}$ be the center and the derived Lie sublattice of~$\g_{L}$, respectively. 
Since~$\ri_L$ is a ring of integers of a number field~$L$, we can choose ordered sets~$\mathbf{e}$ and~$\mathbf{f}$ as the ones of Section~\ref{chcc} such that~$\overline{\mathbf{e}}$ and~$\mathbf{f}$ are bases of~$\g_{L}/\z_{L}$ and~$\g_{L}'$, 
respectively. Let $A(\underline{X})$ and $B(\underline{Y})$ be the commutator matrices of~$\g_{L}$ with respect to~$\mathbf{e}$ and~$\mathbf{f}$; see Definition~\ref{commutator}. 
Consider the functions.
\begin{align*}\widetilde{\rlvlzf_{\G(\Lri)}}(s_1,s_2)&:=1+\Pint_{B}\left((-s_1-2)/2,\ub s_1+s_2+2\ub-h-1\right),\\
\widetilde{\clvlzf_{\G(\Lri)}}(s_1,s_2)&:=1+\Pint_{A}\left(-s_1-1,\ua s_1+s_2+\ua-h-1\right),
\end{align*}
where $\Pint_{B}(\underline{r},t)$ and $\Pint_{A}(\underline{r},t)$ are the integrals given in~\eqref{Vo10as} and~\eqref{Vo10}, respectively.
We have shown in Proposition~\ref{intpadic} that 
\begin{align*}
 \rlvlzf_{\G(\Lri)}(s_1,s_2)&=\frac{1}{1-q^{f(r-s_2)}}\widetilde{\rlvlzf_{\G(\Lri)}}(s_1,s_2)\text{ and }&\\
 \clvlzf_{\G(\Lri)}(s_1,s_2)&=\frac{1}{1-q^{f(z-s_2)}}\widetilde{\clvlzf_{\G(\Lri)}}(s_1,s_2).&
\end{align*}

It is clear that the terms $(1-q^{f(r-s_2)})^{-1}$ and $(1-q^{f(r-s_2)})^{-1}$ are given by rational functions in~$q^f$ and~$q^{-fs_2}$, and that they satisfy functional equations under inversion of~$q^{f}$. Therefore, to prove Theorem~\ref{thmA}, it suffices to show that $\widetilde{\rlvlzf_{\G(\Lri)}}(s_1,s_2)$ and $\widetilde{\clvlzf_{\G(\Lri)}}(s_1,s_2)$ behave uniformly under base extension and satisfy functional equations. 

We first show that the integrands of $\widetilde{\rlvlzf_{\G(\Lri)}}(s_1,s_2)$ and $\widetilde{\clvlzf_{\G(\Lri)}}(s_1,s_2)$ are defined over~$\ri$, that is, that only their  domains of integration vary with the ring~$\Lri$. 

The $\Lri$-bases $\overline{\mathbf{e}}$ and $\mathbf{f}$ are only defined locally, and hence so are the matrices $A(\underline{X})$ and $B(\underline{Y})$.
We must assure that there exist $\Lri$-bases~$\overline{\mathbf{e}}$ and~$\mathbf{f}$ as the ones of Section~\ref{chcc} such that the commutator matrices $A(\underline{X})$ and $B(\underline{Y})$, defined with respect to~$\mathbf{e}$ and~$\mathbf{f}$ are defined over~$\ri$, and hence so are the sets of polynomials $F_j(A(\underline{X}))$ and $F_{2j}(B(\underline{Y}))$.  

Since the matrix $B(\underline{Y})$ is the same as the one appearing in the integrands of~\cite[formula~(2.8)]{StVo14} and $A(\underline{X})$ is obtained in an analogous way,
the argument of~\cite[Section~2.3]{StVo14} also holds in this case. Namely, we choose an $\ri$-basis~$\mathbf{f}$ for a free finite-index $\ri$-submodule of the isolator $i(\Lambda')$ of the derived $\ri$-Lie sublattice of~$\Lambda$; see Section~\ref{chcc}.
By~\cite[Lemma~2.5]{StVo14}, $\mathbf{f}$~can be extended to an $\ri$-basis~$\mathbf{e}$ for a free $\ri$-submodule~$M$ of finite index of~$\Lambda$. 
If the residue characteristic~$p$ of~$\p$ does not divide $|\Lambda:M|$ or $|i(\Lambda'): \Lambda'|$, this basis~$\mathbf{e}$ may be used to 
obtain an $\Lri$-basis for~$\Lambda(\Lri)$, by tensoring the elements of~$\mathbf{e}$ with~$\Lri$. 

\begin{rmk}
The condition ``$p$ does not divide $|i(\Lambda'): \Lambda'|$'' is missing in~\cite{StVo14}, but this omission does not affect the proof of~\cite[Theorem~A]{StVo14}, since this 
condition only excludes a finite number of prime ideals~$\p$.
This was first pointed out in~\cite[Section~3.3]{DuVo14}.
\end{rmk}

We now recall the general integrals given in~\cite[Section~2.1]{Vo10} and show that the integrals~$\widetilde{\rlvlzf_{\G(\Lri)}}(s_1,s_2)$ and~$\widetilde{\clvlzf_{\G(\Lri)}}(s_1,s_2)$ are special cases of such integrals, so that the arguments given in~\cite[Section~4]{AKOV13} assure that they satisfy functional equations.  Recall $[u]=\{1, \dots, u\}$ for $u \in \N$. 

Fix $l,m,n \in \N$. For each $k \in [l]$, let $J_k$ be a finite index set. Fix $I \subseteq [n-1]$. Further, fix non-negative integers~$e_{ikj}$ and finite sets~$F_{kj}(\underline{Y})$ of polynomials over~$\lri$, for $k \in [l], j\in J_k$ and $i \in I$. 
Also, let $\mathcal{W}(\lri)\subseteq \lri^m$ be a union of cosets modulo~$\p^{(m)}$. Define
\begin{equation}\label{Vogen}
 \pints_{\mathcal{W}(\lri),I}(\underline{s})=\int_{\p^{(|I|)}\times \mathcal{W}(\lri)} 
 \prod_{k=1}^{l}\left|\left|\bigcup_{j \in J_k} \left( \prod_{i \in I}x_{i}^{e_{ikj}}\right)F_{kj}(\underline{y})\right|\right|_{\p}^{s_k}d\mu,
\end{equation}
where $\underline{s}=(s_1, \dots, s_l)$ is a vector of complex variables and $\underline{x}=(x_i)_{i \in I}$ and $\underline{y}=(y_1, \dots, y_m)$ are independent integration variables. 

In~\cite[Corollary~2.4]{Vo10}, by studying the transformation of the integral~\eqref{Vogen} under a principalisation~$(Y,h)$ of the ideal $\prod_{k,j}(F_{kj}(\underline{Y}))$, Voll proved a functional equation for~\eqref{Vogen} under inversion of the parameter~$q$ under certain invariance and regularity conditions. In particular, it is required that the principalisation~$(Y,h)$ has good reduction modulo~$\p$, a condition that is satisfied for almost all prime ideals~$\p$. 

We now relate the integrals of Proposition~\ref{intpadic} with the general integral~\eqref{Vogen}.

Set $I=\{1\}$ and write $x_1=x$. Set~$n=b$, $m=b^2$, $l=2\ub+1$, and $J_k=\{1,2\}$, if $k \in [\ub]$ and $J_{k}=\{1\}$ if $\ub<k\leq 2\ub+1$.
We also set $\mathcal{W}(\Lri)=\Gl_b(\Lri)$, and 
\begin{table*}[h]
	\centering
		\begin{tabular}{ c | c | c | c }
				$k$ & $j$ & $F_{kj}$ & $e_{1kj}$ \\
				\hline
				$\leq \ub$ & $1$ & $F_{2k}(B(\underline{y}))$  & $0$ \\
				
				$\ub<k\leq 2\ub$ & $1$ & $F_{2(k-1-\ub)}(B(\underline{y}))$ & $0$ \\
				
				$2\ub+1$ & $1$ & $\{1\}$ & $1$ \\
				$\leq \ub$ & $2$ & $F_{2(k-1)}(B(\underline{y}))$ & 2
		\end{tabular}.
\end{table*}

We see that, with this set-up, the integral~\eqref{Vogen} is equal to
\begin{align}
 \label{intb2}\pints_{\Gl_b(\Lri),\{1\}}(\underline{s})=\int_{\Pp\times \Gl_b(\Lri)} &\|x\|_{\Pp}^{s_{2\ub+1}}\cdot\\
										  \prod_{k=1}^{\ub}\|F_{2k}(B(\underline{Y}))\cup x^2F_{2(k-1)}(B(\underline{Y}))\|_{\Pp}^{s_k}
										  &\prod_{k=\ub+1}^{2\ub}\|F_{2(k-1-\ub)}(B(\underline{Y}))\|_{\Pp}^{s_{k}}d\mu.\nonumber
\end{align}
Set
\[\mathbf{a}_{1}^{\irrup}=(-\frac{1}{2}\mathbf{1}_{\ub},\frac{1}{2}\mathbf{1}_{\ub},\ub), \hspace{0.3cm} \mathbf{a}_{2}^{\irrup}=(\mathbf{0}_{\ub}, \mathbf{0}_{\ub},1),\]
\[\mathbf{b}^{\irrup}=(-\mathbf{1}_{\ub}, \mathbf{1}_{\ub}, 2\ub-h-1),\]
where $\mathbf{1}_{\ub}=(1, \dots, 1)\in \Z^{\ub}$ and $\mathbf{0}_{\ub}=(0,\dots,0)\in\Z^{\ub}$.

Although the domain of integration of the integral~\eqref{intb2} involves $\Gl_b(\Lri)$, the integrand only depends on the entries of the first column, say, as explained in~\cite[Section~4.1.3]{AKOV13}. 
It follows that 
\[\widetilde{\rlvlzf_{\G(\Lri)}}(s_1,s_2)=1+\frac{1}{1-q^{-f}}\left(\prod_{k=1}^{b-1}(1-q^{-fk})\right)^{-1}\hspace{-0.2cm}
\pints_{\Gl_b(\Lri),\{1\}}\left(\mathbf{a}_{1}^{\irrup}s_1+\mathbf{a}_{2}^{\irrup}s_2+\mathbf{b}^{\irrup}\right).\]

Analogously, for~$n=a$, $m=a^2$, one can find appropriate data $l\in \N$, $J_k$, $e_{1jk}$, and $F_{kj}(\underline{X})$ such that 
\[\widetilde{\clvlzf_{\G(\Lri)}}(s_1,s_2)=1+\frac{1}{1-q^{-f}}\left(\prod_{k=1}^{a-1}(1-q^{-fk})\right)^{-1}\hspace{-0.2cm}
\pints_{\Gl_a(\Lri),\{1\}}(\mathbf{a}_{1}^{\ccup}s_1+\mathbf{a}_{2}^{\ccup}s_2+\mathbf{b}^{\ccup}),\]
for $\mathbf{a}_{1}^{\ccup}=(-\mathbf{1}_{\ua},\mathbf{1}_{\ua},\ua)$, $\mathbf{a}_{2}^{\ccup}=(\mathbf{0}_{\ua}, \mathbf{0}_{\ua},1)$, 
$\mathbf{b}^{\ccup}=(-\mathbf{1}_{\ua}, \mathbf{1}_{\ua}, \ua-h-1)$.

Theorem~\ref{thmA} then follows by the arguments given in~\cite[Section~4]{AKOV13}. 

\section*{Acknowledgements}
 I am thankful to my advisor Christopher Voll, and to Tobias Rossmann and Yuri Santos Rego for helpful discussions and comments on this work. I am also grateful to the anonymous referee for valuable comments and suggestions.
 I gratefully acknowledge financial support from the DAAD for this work.
 \def\cprime{$'$} \def\cprime{$'$}
 \providecommand{\bysame}{\leavevmode\hbox to3em{\hrulefill}\thinspace}
 \providecommand{\MR}{\relax\ifhmode\unskip\space\fi MR }

\printbibliography

\end{document}